\documentclass{aims}
\usepackage{amsmath}
  \usepackage{paralist}
 \usepackage[colorlinks=true]{hyperref}
\hypersetup{urlcolor=blue, citecolor=red}

\usepackage{amsmath}
\usepackage{amsthm}
\usepackage{amssymb}
\usepackage{amscd}
\usepackage{latexsym}
\usepackage{amsfonts}
\usepackage[novbox]{pdfsync}

\input xy
\xyoption{all} \tolerance=500

\def\currentvolume{X}
 \def\currentissue{X}
  \def\currentyear{200X}
   \def\currentmonth{XX}
    \def\ppages{X--XX}
     \def\DOI{10.3934/xx.xx.xx.xx}

\newtheorem{theorem}{Theorem}[section]

\newtheorem{proposition}{Proposition}

\theoremstyle{definition}

\newtheorem{remark}{Remark}

\newtheorem{example}{Example}

\font\black=cmbx10 \font\sblack=cmbx7 \font\ssblack=cmbx5 \font\blackital=cmmib10  \skewchar\blackital='177
\font\sblackital=cmmib7 \skewchar\sblackital='177 \font\ssblackital=cmmib5 \skewchar\ssblackital='177
\font\sanss=cmss10 \font\ssanss=cmss9 
\font\sssanss=cmss8 scaled 600
\font\blackboard=msbm10 \font\sblackboard=msbm7 \font\ssblackboard=msbm5
\font\caligr=eusm10 \font\scaligr=eusm7 \font\sscaligr=eusm5  \font\fraktur=eufm10
\font\sfraktur=eufm7 \font\ssfraktur=eufm5 

\font\bsymb=cmsy10 scaled\magstep2
\def\all#1{\setbox0=\hbox{\lower1.5pt\hbox{\bsymb
       \char"38}}\setbox1=\hbox{$_{#1}$} \box0\lower2pt\box1\;}
\def\exi#1{\setbox0=\hbox{\lower1.5pt\hbox{\bsymb \char"39}}
       \setbox1=\hbox{$_{#1}$} \box0\lower2pt\box1\;}

\def\tx#1{{\fam0\relax#1}}

\newfam\bifam
\textfont\bifam=\blackital \scriptfont\bifam=\sblackital \scriptscriptfont\bifam=\ssblackital

\newfam\blfam
\textfont\blfam=\black \scriptfont\blfam=\sblack \scriptscriptfont\blfam=\ssblack

\newfam\bbfam
\textfont\bbfam=\blackboard \scriptfont\bbfam=\sblackboard \scriptscriptfont\bbfam=\ssblackboard

\newfam\ssfam
\textfont\ssfam=\sanss \scriptfont\ssfam=\ssanss \scriptscriptfont\ssfam=\sssanss
\def\sss#1{{\fam\ssfam\relax#1}}

\newfam\clfam
\textfont\clfam=\caligr \scriptfont\clfam=\scaligr \scriptscriptfont\clfam=\sscaligr

\newfam\frfam
\textfont\frfam=\fraktur \scriptfont\frfam=\sfraktur \scriptscriptfont\frfam=\ssfraktur

\def\hpb#1{\setbox0=\hbox{${#1}$}
    \copy0 \kern-\wd0 \kern.2pt \box0}
\def\vpb#1{\setbox0=\hbox{${#1}$}
    \copy0 \kern-\wd0 \raise.08pt \box0}

\def\pmb#1{\setbox0\hbox{${#1}$} \copy0 \kern-\wd0 \kern.2pt \box0}
\def\pmbb#1{\setbox0\hbox{${#1}$} \copy0 \kern-\wd0
      \kern.2pt \copy0 \kern-\wd0 \kern.2pt \box0}
\def\pmbbb#1{\setbox0\hbox{${#1}$} \copy0 \kern-\wd0
      \kern.2pt \copy0 \kern-\wd0 \kern.2pt
    \copy0 \kern-\wd0 \kern.2pt \box0}
\def\pmxb#1{\setbox0\hbox{${#1}$} \copy0 \kern-\wd0
      \kern.2pt \copy0 \kern-\wd0 \kern.2pt
      \copy0 \kern-\wd0 \kern.2pt \copy0 \kern-\wd0 \kern.2pt \box0}
\def\pmxbb#1{\setbox0\hbox{${#1}$} \copy0 \kern-\wd0 \kern.2pt
      \copy0 \kern-\wd0 \kern.2pt
      \copy0 \kern-\wd0 \kern.2pt \copy0 \kern-\wd0 \kern.2pt
      \copy0 \kern-\wd0 \kern.2pt \box0}

\mathchardef\za="710B  
\mathchardef\zb="710C  
\mathchardef\zg="710D  
\mathchardef\zd="710E  
\mathchardef\zve="710F 
\mathchardef\zz="7110  
\mathchardef\zh="7111  
\mathchardef\zvy="7112 
\mathchardef\zi="7113  
\mathchardef\zk="7114  
\mathchardef\zl="7115  
\mathchardef\zm="7116  
\mathchardef\zn="7117  
\mathchardef\zx="7118  
\mathchardef\zp="7119  
\mathchardef\zr="711A  
\mathchardef\zs="711B  
\mathchardef\zt="711C  
\mathchardef\zu="711D  
\mathchardef\zvf="711E 
\mathchardef\zq="711F  
\mathchardef\zc="7120  
\mathchardef\zw="7121  
\mathchardef\ze="7122  
\mathchardef\zy="7123  
\mathchardef\zf="7124  
\mathchardef\zvr="7125 
\mathchardef\zvs="7126 
\mathchardef\zf="7127  
\mathchardef\zG="7000  
\mathchardef\zD="7001  
\mathchardef\zY="7002  
\mathchardef\zL="7003  
\mathchardef\zX="7004  
\mathchardef\zP="7005  
\mathchardef\zS="7006  
\mathchardef\zU="7007  
\mathchardef\zF="7008  
\mathchardef\zW="700A  
\mathchardef\zC="7009  

\newcommand{\be}{\begin{equation}}
\newcommand{\ee}{\end{equation}}

\newcommand{\bea}{\begin{eqnarray}}
\newcommand{\eea}{\end{eqnarray}}
\newcommand{\beas}{\begin{eqnarray*}}
\newcommand{\eeas}{\end{eqnarray*}}
\def\*{{\textstyle *}}
\newcommand{\R}{{\mathbb R}}

\def\ssT{\sss T}

\newcommand{\we}{\wedge}
\newcommand{\nn}{\nonumber}

\newcommand{\s}{{\textstyle *}}

\newcommand{\pa}{\partial}
\newcommand{\ti}{\times}

\def\ix{\operatorname{i}}

\def\cC{\mathcal{C}}

\def\cL{\mathcal{L}}

\def\Z{\mathbf{Z}}

\def\sT{{\sss T}}

\def\sV{{\sss V}}

\def\xd{\tx{d}\,}
\def\xi{\tx{i}}

\def\dt{\xd_{\textsf{T}}}

\newdir{ (}{{}*!/-5pt/@^{(}}

\def\eza{\za}
\def\ezb{\zb}

\def\ezk{\zk}

\def\ezp{\zp}

\def\ezt{\zt}
\def\ezw{\zw}

\def\ezvy{\zvy}

\def\ix{\operatorname{i}}
\def\xit{\ix_{\textsf{T}}}
\def\xd{\operatorname{d}\!}

\def\dt{\xd_{\textsf{T}}}
\def\bdt{\bar\xd_{\,\textsf{T}}}
\def\s*{{\scriptstyle *}}

\def\geqs{\geqslant}

\title[Dynamics of strings] 
      {Geometry of Lagrangian and Hamiltonian formalisms\\ in the dynamics of strings}

\author[ Janusz Grabowski, Katarzyna Grabowska, Pawe\l\ Urba\'nski]{}

\subjclass{Primary: 70S05, 5730; Secondary: 53D05, 58A20, 58A32.}
 \keywords{Tulczyjew triples, Lagrange formalism, Hamiltonian formalism,
variational calculus, double vector bundles, minimal surfaces.}

 \email{jagrab@impan.pl}
 \email{konieczn@fuw.edu.pl}
 \email{urbanski@fuw.edu.pl}

\thanks{Research  founded by the  Polish National Science Centre grant under the contract number DEC-2012/06/A/ST1/00256.}

\begin{document}
\maketitle

\centerline{\scshape Janusz Grabowski }
\medskip
{\footnotesize
 \centerline{Institute of Mathematics, Polish Academy of Sciences,}
   \centerline{ \'Sniadeckich 8, 00-656 Warszawa, POLAND}
} 

\medskip

\centerline{Katarzyna Grabowska and Pawe\l\  Urba\'nski}
\medskip
{\footnotesize
 \centerline{ Division of Mathematical Methods in Physics,}
   \centerline{University of Warsaw, Ho\.za 74, 00-682 Warszawa, POLAND}
}

\bigskip

 \centerline{(Communicated by the associate editor name)}

\begin{abstract}
{The Lagrangian description of mechanical systems and the Legendre Transformation
(considered as a passage from the Lagrangian to the Hamiltonian
formulation of the dynamics) for point-like objects, for which the
infinitesimal configuration space is $\textsf{T} M$, is based on the existence of
canonical symplectic isomorphisms of double vector bundles $\textsf{T}^\*\textsf{T}M$, $\textsf{T}^\*\textsf{T}^\* M$,
and $\textsf{T}\textsf{T}^\* M$, where the symplectic structure on $\textsf{T}\textsf{T}^\* M$ is the tangent lift of the canonical symplectic structure $\textsf{T}^\* M$.
We show that there exists an analogous picture in the dynamics of objects for which the configuration space is $\wedge^n \textsf{T} M$, if we make use of certain structures of graded bundles of degree $n$, i.e. objects generalizing vector bundles (for which $n=1$). For instance, the role of $\textsf{T}\textsf{T}^\*M$ is played in our approach by the manifold $\we^n\textsf{T} M\we^n\textsf{T}^\*M$, which is canonically a graded bundle of degree $n$ over $\we^n\textsf{T} M$. Dynamics of strings and the Plateau problem in statics are particular cases of this framework.}
\end{abstract}

\section{Introduction}\label{s:1}
This work is an account of a research undertaken jointly with W.~M.~Tulczyjew on the Legendre transformation in the dynamics of strings.
{Let us recall} structures which are involved in the Legendre transformation for the dynamics of  point-like objects, i.e. when the motion is given by a one-dimensional submanifold (parametrized or not) in the configuration manifold $M$. An infinitesimal piece of motion is a tangent vector and a system is described by a Lagrangian, i.e. a function (constrained or not)  on $\sT M$.  The geometric content of Lagrangian and Hamiltonian formulation of the dynamics is contained in the following diagram.
    \begin{equation}\label{F3.1}
    \xymatrix@R-3mm @C-7mm
        { & \sT^\*\sT^\*M \ar[ldd]_*{} \ar[rd]^*{} & & & \sT \sT^\* M \ar[rrr]^*{{\eza}_M}
        \ar[lll]_*{\ezb_M} \ar[ldd]^*{} \ar[rd]^*{}& & & \sT^\*\sT M \ar[ldd]^*{} \ar[rd]^*{} & \cr
        & & \sT M \ar[ldd]^*{} & & & \sT M  \ar[ldd]^*{} \ar[lll]^*{} \ar[rrr]^*{} & & & \sT M \ar[ldd]^*{}  \cr
         \sT^\* M  \ar[rd]^*{}  & & & \sT^\* M \ar[rrr]^*{} \ar[lll]^*{} \ar[rd]^*{}& & & \sT^\* M
        \ar[rd]^*{} & & \cr
        & M  & & &  M \ar[rrr]^*{} \ar[lll]^*{} & & & M & }.
\end{equation}
The dynamics is a subset of $\sT\sT^\* M$, which can be transported either to $\sT^\*\sT M$  (Lagrangian side) or to $\sT^\*\sT^\* M$ (Hamiltonian side). The top elements in the diagram are certain canonical {\it double vector bundles} over $\sT M$ and $\sT^\* M$. A double vector bundle is, roughly speaking, a manifold with two compatible vector bundle structures, i.e. two vector bundle structures whose Euler vector fields commute (see Section (\ref{s:4})).

The mappings $\ezb_M$ and $\eza_M$ are double vector bundle isomorphisms. Beside being double vector bundles, each of the manifolds: $\sT^\*\sT^\* M, \ \sT^\*\sT M$, and   $\sT\sT^\* M$ is, in a canonical way,  a symplectic manifold. The bundles $\sT^\*\sT^\* M, \ \sT^\*\sT M$ carry the canonical symplectic structures of cotangent bundles and $\sT\sT^\* M$ is equipped with the tangent lift $\xd_\sT \ezw_M$ of the canonical symplectic form $\ezw_M$ on $\sT^\* M$. With these symplectic structures, $\ezb_M$ and $\eza_M$ are symplectomorphisms  (cf.\cite{Tu1,Tu6,Tu3}). The mapping $\zb_M$ is given by the contraction of the canonical symplectic form on $\sT^\* M$. The mapping $\eza_M$ can be defined as the composition of $\ezb_M$ with the canonical isomorphism of $\sT^\* \sT M$ and $\sT^\* \sT^\* M$. This is a particular case of a canonical isomorphism  $\sT^\* E\simeq\sT^\* E^\*$ for any vector bundle $E$. However, it is more appropriate to introduce $\eza_M$ as the map dual to the canonical flip $\ezk_M\colon\sT\sT M \rightarrow \sT\sT M$, which interchanges the canonical and the tangent projections onto $\sT M$. The dual to the bundle $\ezt_{\textsf{T} M}\colon \sT\sT M\rightarrow \sT M$ is of course  $\ezp_{\textsf{T} M}\colon \sT^\*\sT M\rightarrow \sT M$ and the dual to $\sT\ezt_{ M}\colon \sT\sT M\rightarrow \sT M$ can be identified with $\sT{\zp_M}\colon \sT\sT^\* M\rightarrow \sT M$. The pairing related to this identification is the complete lift
    \begin{equation}\label{F3.2}
        \xd_\textsf{T} \zd_M \colon \sT(\sT M \times_M \sT^\* M) = \sT\sT M \times_{\sT M} \sT\sT^\* M \rightarrow \R
    \end{equation}
of the canonical pairing  $ \zd_M \colon \sT M \times_M \sT^\* M \rightarrow \R$.
For more details see \cite{KU}, \cite{LM}, \cite{PT}, \cite{TU}, and \cite{Ur}. Similar constructions in the context of field theory one can find e.g. in \cite{G1,G2,GG2}, while the algebroid version is discussed in \cite{GG,GG1,GGU2}.

\medskip
In the following, we shall give similar constructions, motivated by the study of dynamics of one dimensional, non-parametrized  objects (called frequently \emph{strings}), in which the bundle $\we^2\sT M$ of tangent bivectors replaces $\sT M$. The `time-evolution' of a string adds one dimension, so the motion of a system  is given by a two-dimensional submanifold in the manifold $M$. An infinitesimal piece of the motion is the jet of the submanifold which for first-order theories it is the first jet. We immediately realize that, for this choice of infinitesimal motions, the infinitesimal action (Lagrangian) is not a function on first jets, but a section of certain line bundle over first jets manifold, the dual bundle of the bundle of "first jets with volumes". This leads to essential complications even in one dimensional case (relativistic particle), see \cite{Lu} and an alternative approach in \cite{S}.

\medskip
A compromise is to take for the space of infinitesimal pieces of motions the space of simple non-zero 2-vectors. Such 2-vectors represent first jets of 2-dimensional submanifolds (being 2-dimensional subspaces in tangent spaces) together with a volume which should be taken so that it takes the value 1 on the 2-vector. It is convenient to extend this space to all 2-vectors, i.e. to  the vector bundle $\wedge^2\sT M$. In this setting, a Lagrangian is a positive homogeneous function on $\wedge^2\sT M$ (thus the action functional does not depend on the parametrization of the surface) and the corresponding Hamiltonian (if it exists) is a function on the dual vector bundle  $\wedge^2\sT^\* M$.  The dynamics should be an equation (possibly implicit) for 2-dimensional submanifolds in the phase space, i.e. a subset $D$ in $\wedge^2 \sT \wedge^2\sT^\* M$. A surface $S$ in the phase space $\wedge^2\sT^\* M$ is a solution for the phase dynamics if and only if its tangent space at $\zw\in\wedge^2\sT^\* M$ is represented by a bivector from $D_\zw$. If we use a parametrization, then the tangent bivectors associated with this parametrization must belong to $D$.

\medskip
We realize immediately, that some of the tools, suitable for point-like objects, here  are not adequate. On one hand, $\sT^\*\wedge^2\sT^\* M$ and   $\sT^\*\wedge^2\sT M$ are symplectic manifolds with double vector bundle structures, but it is not true for  $\wedge^2 \sT\wedge^2\sT^\* M$. In general, for a vector bundle $\zt \colon E\rightarrow N$, the fibration $\wedge^2 \sT \zt\colon \wedge^2 \sT E \rightarrow \wedge^2 \sT N $ is not a vector bundle. Moreover, we do not have an identification of  $\wedge^2\sT (E\times_N F)$ with $\wedge^2\sT E \times_{\wedge^2\sT N}\wedge^2\sT F$  and the total derivative $\xd_\textsf{T}$ seems to be useless. Stressing once more: an evolution of a string in $M$ is a surface in $M$ and the (implicit) phase dynamics of the string is represented by a collection of `infinitesimal surfaces' in the phase space, i.e. by a subset of $\wedge^2 \sT \wedge^2\sT^\* M$.

To obtain the dynamics, we use the canonical multisymplectic 3-form $\ezw^2_M$ on $\wedge^2\sT^\* M$, which gives, by the contraction, canonical morphism $\ezb^2_M \colon \wedge^2 \sT\wedge^2\sT^\* M \rightarrow \sT^\*\wedge^2\sT^\* M $. We can replace the total derivative $\xd_\textsf{T}$ by its "higher analogue" $\xd_\textsf{T}^2$, and the bundle $\wedge^2\sT \wedge^2\sT^\* M$ by its quotient bundle {\it modulo} its subbundle $\sV^2 \wedge^2\sT^\* M$ of 2-vertical vectors. These are the ideas we owe to W.~M.~Tulczyjew. What is more, the bundle $\wedge^2 \sT\wedge^2\sT^\* M$ can be viewed as \emph{double graded bundle} in the sense of Grabowski and Rotkiewicz \cite{GR2}, so we get a picture completely analogous to (\ref{F3.1}) except for the fact that the morphisms are not isomorphisms. In this sense, our geometric approach to classical strings is based on morphisms of double graded bundles, although the canonical multisymplectic structure is behind. Let us note also that the graded bundles appear in the supergeometric context under the name {\it graded manifolds} or {\it N-manifolds} \cite{Roy, Sev, Vor}, in particular in the Roytenberg's picture for {\it Courant algebroids} \cite{Roy}.

We want to stress that our framework is completely covariant and, what is more, not reduced to derivation of the Euler-Lagrange equations.
We present the full picture, determining the phase space and the phase equations as the meeting point of the Lagrange and the Hamilton formalisms, subject to the corresponding \emph{Legendre transformation}. The equations are obtained purely geometrically by means of the morphisms in the triple from the lagrangian submanifolds generated by Lagrangians or Hamiltonians. In particular, on the Hamiltonian side we do not use, at least explicitly, any Poisson brackets.

Note finally that classical field theory is usually associated with the concept of a multisymplectic structure.  The multisymplectic approach appeared first in the papers of the `Polish school' \cite{Ga,KS,KTu,Tu5}. Then, it was developed by Gotay, Isennberg, Marsden, and others in \cite{GIMa,GIMb}. The original idea of the multisymplectic structure has been thoroughly investigated and developed by many contemporary authors, see e.g. \cite{CIL,CCI1,EM,FP1,FP2}. The Tulczyjew triple in the context of multisymplectic field theories appeared recently in \cite{CGM} and \cite{LMS} (see also \cite{V}). A similar picture, however with differences on the Hamiltonian side, one can find in \cite{GM} (see also \cite{GMS,Kr}) and many others; it is not possible to list all achievements in this area.

As we use the canonical multisymplectic form on $\we^n\sT^\ast M$, the prototype of all the abstract definitions of such structure, our paper can be viewed also as an attempt to start again the discussion about what is the reasonable concept of an abstract multisymplectic structure (cf. \cite{CIL, FG, Gun, Ma}). We want to undertake this discussion in a separate paper.

\section{Notation and local coordinates}\label{s:2}
Let $M$ be a smooth manifold. We denote by $\ezt_M\colon \sT M\rightarrow M$ the tangent vector bundle and by $\ezp_M\colon \sT^\* M\rightarrow M$ the cotangent vector bundle. Let $\zt\colon E \rightarrow M$ be a vector bundle and $\zp \colon E^\* \rightarrow M$ the dual bundle. We use the following notations for tensor bundles:
    \begin{align*}
        \zt^{\otimes k}&\colon E\otimes_M \cdots \otimes _M E  = \otimes^k_M(E) \longrightarrow M , \\
        \zt^{\wedge  k}&\colon E\wedge_M \cdots \wedge_M E = \wedge^{k}_M (E) \longrightarrow M ,
    \end{align*}
the modules of sections over $\cC^\infty(M)$:
    \be
        \otimes ^k(\zt)= \zG(\otimes ^k_M(E))\,,  \quad
        \zF^k(\zt)= \zG(\wedge ^k_M(E))\,,
    \ee
and the corresponding tensor algebras
    \be
        \otimes (\zt)= \oplus _{k\in \Z} \otimes ^k(\zt)\,,  \quad
        \zF(\zt)=  \oplus_{k\in \Z} \zF^k(\zt)\,.
    \ee
In particular,
    $$\otimes ^0(\zt) = \zF^0(\zt)  =\cC^\infty (M)$$
and
    $$ \otimes ^k(\zt)= \zF^k(\zt)  = \{0\} \ \ \text{for $k<0$}.$$
For a differentiable mapping $\zC\colon M\rightarrow N$,
there are defined higher tangent maps
$$\otimes^k\sT\zC \colon \otimes^k\sT M \rightarrow  \otimes ^k\sT N\,,\quad
     \otimes ^k\sT\zC(v_1\we\cdots\we v_k)=\sT\zC(v_1)\otimes\cdots\otimes\sT\zC(v_k)$$
and
     $$\wedge ^k\sT\zC \colon \wedge ^k\sT M \rightarrow  \wedge ^k\sT N\,,\quad
     \wedge ^k\sT\zC(v_1\we\cdots\we v_k)=\sT\zC(v_1)\we\cdots\we\sT\zC(v_k)\,.$$
By $\sV(E)$ we denote the subbundle of tangent vertical vectors:
    $$  \sV(E) = \{v\in \sT E \colon  \sT\zt(v) =0 \} .$$
We introduce also the bundle $\sV^1(E)$ of 1-vertical bivectors
    $$ \sV^1(E) = \{u\in \wedge ^2\sT E\colon \wedge ^2\sT \ezt (u) =0  \}$$
and its subbundle of 2-vertical bivectors
    $$ \sV^1(E) \supset \sV^2(E) = \{u\in \wedge ^2\sT E\colon u \in \wedge^2_E\sV(E)\}.$$
Let $(x^\zm), \ \zm=1,\dots,n$, be a coordinate system in $M$. We introduce the induced coordinate systems
    \begin{align*}
        (x^\zm, {\dot x}^\zn) &\ \text{in} \ \sT M,\\
        (x^\zm, p_\zn) &\ \text{in} \ \sT^\* M,\\
        (x^\zm, {\dot x}^{\zn\zl}) &\ \text{in} \ \wedge ^2\sT M,\\
        (x^\zm, p_{\zn\zl})&\ \text{in} \ \wedge ^2\sT^\* M\,.
    \end{align*}
Formally, we should take only $\zn<\zm$, but we can use all pairs of indices, assuming  $p_{\zm\zn} = - p_{\zn\zm}$ and ${\dot x}^{\zm\zn} = - {\dot x}^{\zn\zm}$. With this convention, we have the following representation of bivectors and bi-covectors:
    \begin{equation}\label{F0.1}
        \frac{1}{2} {\dot x}^{\zm\zn}\frac{\partial }{\partial x^\zm}\wedge \frac{\partial }{\partial x^\zn}, \ \
        \frac{1}{2} p_{\zm\zn} \xd x^\zm\wedge \xd x^\zn.
    \end{equation}
Let $\zt\colon  E\rightarrow M$ be a vector bundle and let $(x^\zm, e^a)$ be an affine coordinate system on $E$, consisting of basic functions $x^\zm$ (constant along fibers) and functions $e^a$ linear along fibers. We introduce the adapted coordinate systems
    \begin{align*}
        (x^\zm, \zx_a), \quad &\quad \text{in} \ E^\* ,\\
        (x^\zm, e^a,{\dot x}^\zn, {\dot e}^b ) &  \quad \text{in} \ \sT E,\\
        (x^\zm, \zx_a, {\dot x}^\zn, {\dot \zx}_b) & \quad \text{in} \ \sT E^\* ,\\
        (x^\zm, e^a, p_\zn, \zp_b) & \quad \text{in}\ \sT^\*E, \\
        (x^\zm, \zx_a, p_\zn, \zf^b) & \quad \text{in}\ \sT^\* E^\* .
    \end{align*}
In $\sT\sT M$, we use the  coordinates $(x^\zm, {\dot x}^\zn, {x'}^\zl, {\dot x'}{}^\zk)$. Vectors on $\wedge ^2\sT M$ and $\wedge^2\sT^\* M$ read
    \begin{equation}\label{F0.2}\begin{split}
        \sT \wedge ^2\sT M \ni v&= {x'}^\zm \frac{\partial }{\partial x^\zm} + \frac{1}{2}{\dot x'}{}^{\zm\zn}
        \frac{\partial }{\partial {\dot x}^{\zm\zn}},    \\
        \sT \wedge ^2\sT^\* M \ni v&= {\dot x}^\zm \frac{\partial }{\partial x^\zm} + \frac{1}{2}{\dot p}_{\zm\zn}
        \frac{\partial }{\partial p_{\zm\zn}},
    \end{split}\end{equation}
where
$${\dot x'}{}^{\zm\zn} =-{\dot x'}{}^{\zn\zm}, \;
{\dot p}_{\zm\zn} = - {\dot p}_{\zn\zm}, \;
\frac{\partial }{\partial {\dot x}^{\zm\zn}} = -\frac{\partial }{\partial {\dot x}^{\zn\zm}}, \;
\frac{\partial }{\partial p_{\zm\zn}} = - \frac{\partial }{\partial  p_{\zn\zm}}.  $$
Consequently, the following representations of bivectors will be used:

    \bea\label{F0.3}
        \wedge ^2\sT E\ni u &=& \frac{1}{2} {\dot x}^{\zm\zn} \frac{\partial}{\partial x^\zm}\wedge \frac{\partial }{\partial x^\zn} + y^{\zm a}
        \frac{\partial }{\partial x^\zm}\wedge \frac{\partial }{\partial e^a}+\frac{1}{2} {\dot e}^{ab} \frac{\partial }{\partial e^a}\wedge
        \frac{\partial }{\partial e^b}\,,  \nn   \\
        \wedge ^2\sT \sT M\ni u  &=& \frac{1}{2} {\dot x}^{\zm\zn} \frac{\partial}{\partial x^\zm}\wedge \frac{\partial }{\partial x^\zn} + y^{\zm \zn}\frac{\partial }{\partial x^\zm}\wedge \frac{\partial }{\partial {\dot x}^\zn} +\frac{1}{2} {z}^{\zm\zn} \frac{\partial }{\partial {\dot x}^\zm}\wedge \frac{\partial }{\partial {\dot x}^\zn}\,,  \\ \nn
        \wedge ^2\sT \wedge^2 \sT^\* M\ni u &=&\begin{aligned}[t] \frac{1}{2} {\dot x}^{\zm\zn}\frac{\partial }{\partial x^\zm}\wedge \frac{\partial }{\partial x^\zn} + \frac{1}{2} y^{\zm}_{\zn\zl} &\frac{\partial }{\partial x^\zm}\wedge \frac{\partial }{\partial p_{\zn\zl}} +
        \frac{1}{8} {\dot p}_{\zm\zn\zl\zk} \frac{\partial }{\partial p_{\zm\zn}}\wedge\frac{\partial }{\partial p_{\zl\zk}}\,,\end{aligned}
    \eea
where
    $$y^\zm_{\zn\zl} = - y^\zm_{\zl\zn}, \ {\dot p}_{\zm\zn\zl\zk} = -{\dot p}_{\zn\zm\zl\zk} = - {\dot p}_{\zm\zn\zk\zl} = - {\dot p}_{\zl\zk\zm\zn}.$$

For each $k$, we define a canonical $k$-form $\ezvy^k_M$ on $\wedge ^k \sT^\* M$ by the formula
    $$\ezvy^k_M(v_1,\dots,v_k)= p(\sT\ezp^{\wedge k}_M v_1,\dots,\sT\ezp^{\wedge k}_M v_k),$$
where $v_i\in \sT_p \wedge ^k \sT^\* M$ and $\ezp^{\wedge k}_M \colon \wedge ^k \sT^\* M \rightarrow M$ is the canonical projection. The form $\ezvy^k_M$ is called the {\it Liouville $k$-form} and its exterior differential $\ezw^k_M =\xd \ezvy^k_M$ is the {\it canonical multisymplectic form} on  $\wedge ^k \sT^\* M$.

In local coordinates, the Liouville 2-form on $\wedge ^2 \sT^\* M$ is given, by the formula
    $$ \ezvy^2_M = \frac{1}{2} p_{\zm\zn} \xd x^\zm\wedge \xd x^\zn,$$
and  the  canonical multisymplectic 3-form reads
    $$  \ezw^2_M = \xd \ezvy_M^2 =   \frac{1}{2} \xd p_{\zm\zn} \wedge\xd x^\zm\wedge \xd x^\zn \,.$$

\section{Some important differential operators}\label{s:3}
The {\it total derivative} (the {\it complete lift}) $\dt$ on a manifold $N$ is a graded derivation of degree 0 from $\zF(\ezp_N)$ to $\zF(\ezp_{\textsf{T} N})$ with respect to the pull-back homomorphism $\ezt_N^\*\colon\zF(\ezp_N) \to \zF(\ezp_{\textsf{T} N})$ (see \cite{PT}). A standard definition  is given by the graded commutator $\dt =  \xit \xd + \xd\xit$, where $\xit$ is a derivation $\xit\colon\zF(\ezp_N)\rightarrow\zF(\ezp_{\textsf{T} N})$ of degree -1, defined by the formula
    $$ \xit \zf(v_1, v_2, v_3,\dots, v_{p-1}) = \zf(u, \sT\ezt_N(v_1), \dots,\sT \ezt_N(v_{p-1}))\,, $$
where $v_i\in\sT_u \sT N$. By analogy, we define (cf. \cite{Bu}) a  graded differential  operator
    $$\dt^2\colon\zF(\ezp_N)\rightarrow\zF(\ezp_{\wedge ^2\textsf{T} N})$$
of degree -1 as the graded commutator
    \begin{equation}\label{F0.6}
        \dt^2 = \xd\, \xit^2 - \xit^2 \xd\,
    \end{equation}
where $\xit^2$ is a graded differential operator  of degree -2
    \begin{equation}\label{F0.4}\begin{split}
        \xit^2 &\colon \zF(\ezp_N) \rightarrow  \zF(\ezp_{\wedge ^2\textsf{T} N}) \\
        &\colon  \zF^p(\ezp_N) \rightarrow  \zF^{p-2}(\ezp_{\wedge ^2\textsf{T} N}),
    \end{split}\end{equation}
defined by the formula
    \begin{equation}\label{F0.5}
        \xit^2 \zf(v_1, v_2, v_3,\dots, v_{p-2}) = \zf(u, \sT \ezt^{\wedge 2}_N(v_1), \dots,\sT \ezt^{\wedge 2}_N(v_{p-2}))\,,
    \end{equation}
for $v_i\in \sT_u\wedge ^2\sT N$, $u\in \wedge ^2\sT N$.

It is obvious that $\dt$ commutes (in the graded sense) with the exterior derivative:
     \be\label{dd}\xd \dt^2 = \xd\, (\xd\xit^2 -\xit^2\xd\,) = -\xd\,\xit^2\xd\, + \xit^2\xd\xd\, =-(\xd\xit^2 -\xit^2\xd\,)\xd =-\dt^2\xd \,.
     \ee
In local coordinates, for a 1-form $\zf = \zf_\zm\xd x^\zm$, we have
$\xit^2 \zf = 0$ and
    \begin{equation}\label{F0.10}
        \dt^2 \zf = -\xit^2 \xd \zf = -\xit^2\left(\frac{1}{2}\left(
        \frac{\partial \zf_\zm}{\partial x^\zn} -  \frac{\partial
        \zf_\zn}{\partial x^\zm}\right) \xd x^\zn\wedge \xd x^\zm \right)=
        \frac{\partial \zf_\zm}{\partial x^\zn} {\dot x}^{\zm\zn}.
    \end{equation}
It is easy to see that, for a differentiable mapping $\zC\colon M\rightarrow N$,
    \begin{equation}\begin{split}\label{F0.7}
        \xit^2 \zC^\* &= \left(\wedge ^2\sT\zC\right)^\* \xit^2, \\
        \dt^2 \zC^\* &= \left(\wedge ^2\sT\zC\right)^\* \dt^2\,,
    \end{split}\end{equation}
and that a $p$-form $\za$, $p>1$, on $N$ is closed if and only if
\be\label{v} \tilde \za^\* \ezw^{p-2}_N = \dt^2\za\,,
\ee
where
    $$ \tilde \za \colon \wedge ^2\sT N \rightarrow  \wedge ^{p-2}\sT^\* N \colon v\mapsto \za(v,\cdot)$$
and $\ezw^{p-2}_N$ is the canonical $(p-1)$-form on  $\wedge ^{p-2}\sT^\* N$.
These properties correspond to known properties of $\xit$ and $\dt$ (\cite{GU}).

\section{Double vector bundles}\label{s:4}
A double vector bundle is a manifold with two compatible vector bundle structures. The compatibility conditions, in its original formulation (see \cite{KU, Pr}) are complicated, but in a recent work \cite{GR} one can find a very simple formulation of the compatibility condition. It is based on the observation that a vector bundle structure is encoded in its Euler vector field or, equivalently, in a multiplicative monoid $\R_{\geqs 0} = \{r\in \R \colon r\geqs 0\}$ action on the total vector bundle manifold. In local coordinates, the Euler vector field $X_E$ for a vector bundle $\zt\colon E\rightarrow M$ is given by
    $$ X_E = e^a\frac{\partial}{\partial e^a}.$$
A vector bundle morphism $\zf\colon E\rightarrow F$ is characterized by the property $\sT\zf X_E (e) = X_F(\zf(e))$ for each $e\in E$.
Two vector bundle structures are {\it compatible} if the corresponding Euler vector fields commute (equivalently, the monoid actions commute). A manifold with two compatible vector bundle structures is called a {\it double vector bundle}. A double vector bundle structure on a manifold $K$ is represented by the commutative diagram of vector bundle morphisms:
    \begin{equation}\label{F03.1}
        {\xymatrix@R-1mm @C-1mm{ & K \ar[ld]_*{\zt_l} \ar[rd]^*{\zt_r} & \cr
        E_l \ar[rd]_*{\bar\zt_r} & & E_r \ar[ld]^*{\bar\zt_l}  \cr & M  & }} .
    \end{equation}
For a double vector bundle, a zero section  of one vector bundle structure is a subbundle of the second one. The  projection of one vector bundle structure is a vector bundle morphism for the second one. In other words, the  pairs $(\zt_r,\bar\zt_r)$ and $(\zt_l,\bar\zt_l)$ are vector bundle morphisms. The intersection of kernels of vector bundle projections $\zt_r, \zt_l$ is a vector bundle over $M$, called {\it the core} of the double vector bundle. In other words, the core is a submanifold, where two Euler fields coincide.

A mapping of double vector bundles is a {\it double vector bundle morphism} if it is a vector bundle morphism for both vector bundle structures. It follows that the image of the core of the domain is contained in the core of the co-domain. Note that the tangent and cotangent bundle of a vector bundle carry canonical double vector bundle structures.

\begin{example}\label{C0.1a}{\rm
Let $\zt\colon E\rightarrow M$ be a vector bundle with the Euler vector field $X_E$. The manifold $\sT E$ has two vector bundle structures with  Euler vector fields
   $$ X_{\textsf{T} E} = \dot x^\zm \frac{\partial}{\partial \dot x^\zm} + \dot e^a\frac{\partial}{\partial \dot e^a} \ \ \ \text{(the vector bundle structure of $\ezt_{E}\colon \sT E \rightarrow E$)}$$
and
    $$X^1_{\textsf{T} E} = e^a \frac{\partial}{\partial  e^a} + \dot e^a\frac{\partial}{\partial \dot e^a} \ \ \ \text{(the vector bundle structure of $\sT\zt\colon \sT E \rightarrow \sT M$)}.$$
The core is the vector bundle of vertical vectors at the zero section of $E$ ($e^a=0, \dot x^\zm=0$).
}\end{example}

\begin{example}\label{C0.1b}{\rm
Less obvious is a double vector bundle structure in $\sT^\*E$ with Euler vector fields
   $$ X_{\textsf{T}^\* E} = p_\zm \frac{\partial}{\partial p_\zm} +  \zp^a\frac{\partial}{\partial  \zp^a} \ \ \ \text{(the vector bundle structure of $\ezp_{E}\colon \sT^\* E \rightarrow E$)}$$
and
   $$X^1_{\textsf{T}^\* E} = e^a \frac{\partial}{\partial  e^a} +  p_\zm\frac{\partial}{\partial  p_\zm} \ \ \ \text{(the vector bundle structure of $\sT^\*\zt\colon \sT^\* E \rightarrow E^\*$)  }.$$
The core is the bundle of co-vectors at the zero section, vanishing on vertical vectors\newline ($e^a=0, \zp^a=0$). It can be identified with  $\sT^\* M$.
}\end{example}

A double vector bundle has its dual vector bundles with respect to the right and to the left vector bundle structures. It appears that right and left dual bundles are also double vector bundles. For a double vector bundle (\ref{F03.1}) with the core bundle $C$, the diagram for the right dual looks as follows
    \begin{equation}\label{F03.2}
        {\xymatrix@R-1mm @C-1mm{ & K^{\*_r} \ar[ld]_*{\zp_l} \ar[rd]^*{\zp_r} & \cr
        E_r \ar[rd]^*{\bar\zt_r} & E_l^\* \ar[u] & C^\* \ar[ld]^*{\bar\zt_l}  \cr & M  & }} .
    \end{equation}
$E_l^\*$ is the core of $K^{\*_r}$. For  details see \cite{KU}.
\begin{remark} The Euler vector fields $X^1_{\textsf{T} E}$ and $X^1_{\textsf{T}^\* E}$ can be covariantly described, respectively, as the \emph{tangent} and the \emph{phase lift} of the Euler vector field $X= e^a \frac{\partial}{\partial  e^a}$ of the vector bundle $E$. For details we refer to \cite{GR}.
\end{remark}

\section{Bundle structures on bivectors}\label{s:5}
 A vector space (vector bundle) structure  is uniquely determined by the Euler vector field, however it can be also introduced by a pairing, i.e. by constructing the space of linear functions. We use this method to discuss vector bundle structures related to the fibration $\wedge ^2\sT E$ over $\wedge ^2\sT M$. In the next section we shall discuss also the Euler vector field approach.

It is easy to realize that the fibration $\wedge ^2\sT E\rightarrow \wedge ^2\sT M$ is not a canonical vector bundle structure. We will show that such structure exists in the quotient  $ \wedge ^2 \sT E \big/ \sV^2(E)$. In order to find a proper pairing, we begin  with a co-vector valued pairing.

Let $\zt  \colon E \rightarrow M $ and $\zz \colon F\rightarrow M$ be vector bundles and let
    $$ \zd \colon  E\times _M F\rightarrow \sT^\* M$$
be a bilinear vector bundle mapping. $\zd$ can be considered a semi-basic
1-form on $E\times _MF$. In local coordinates,
    $$  \zd = \zd_{\zm ai}(x)e^af^i \xd x^\zm,$$
where
$(x^\zm, e^a)$, $(x^\zm, f^i)$ are coordinates on $E,\,F$ respectively and $\zd_{\zm ai}$ are functions on $M$.
Using the formula (\ref{F0.10}) and coordinates as in (\ref{F0.3}),  we obtain local expression for the function
    $$\dt^2 \zd \colon \wedge ^2 \sT (E\times_MF) \rightarrow \R\,,$$
    namely
    \begin{equation}\label{F0.11}\begin{split}
        \dt^2 \zd &= -\xit^2 \left(\frac{\partial \zd_{\zm ai}}{\partial
        x^\zl}e^af^i \xd x^\zl\wedge \xd x^\zm  + \zd_{\zm ai} e^a \xd f^i \wedge \xd
        x^\zm  + \zd_{\zm ai} f^i \xd e^a \wedge \xd x^\zm \right)   \\
                        &= \frac{1}{2}\left( \frac{\partial \zd_{\zm ai}}{\partial
        x^\zl} - \frac{\partial \zd_{\zl ai}}{\partial
        x^\zm}  \right) {\dot x}^{\zm\zl} e^a f^i -\zd_{\zm ai}(e^a y^{\zm i} +
        f^i y^{\zm a}).
    \end{split}\end{equation}
We see that this formula does not contain double-vertical forms $\xd e^a\we\xd e^b$, $\xd e^a\we\xd f^i$, and $\xd f^i\we \xd f^j$. It follows that the function $\dt^2\zd$ can be projected to a function on $\wedge ^2\sT (E\times _MF)\big/ \sV^2(E\times _MF)$.
The projected function we denote with $\bdt^{\,2}\zd$.
\begin{remark}

Let $\zt \colon V\rightarrow M$ be a fibration. A form $\zf$ on $V$ is called {\it semi-basic} if, for each vertical vector $v\in \sV(V)$, we have $\xi_v\zf = 0$. In other words, a semi-basic form is a combination of forms $f\zt^\* \za$, where $f$ is a function on $V$. In local coordinates $(x^\zm,e^a)$, a semi basic form contains expressions $\za_{i_1\cdots i_k}(x,e) \xd x^{i_1}\cdots\xd x^{i_k}$  only. In this terminology, $\dt^2\zd$ turns out to be semi-basic on $V=E\ti_MF$.

One can prove that, in general, if $\zf$ is a semi-basic form on a fibration
$\zt \colon V\rightarrow M$, then $\xit^2 \zf$ and $\dt^2 \zf$ are semi-basic with respect to the fibration
    $$
        \wedge ^2 \sT\zt \colon \wedge ^2 \sT V \rightarrow \wedge ^2\sT M\,.
    $$
\end{remark}
\begin{proposition}\label{C0.3}
The bundles
\begin{equation}\label{F0.12}
    \wedge ^2\sT (E\times _MF)\big/ \sV^2(E\times _MF)\ \
    \text{and}\ \  \left(\wedge ^2\sT E \big/ \sV^2(E) \right)
    \times_{\wedge ^2\sT M} \left(\wedge ^2\sT F \big/ \sV^2(F) \right)
\end{equation}
are  canonically isomorphic.
\end{proposition}
\noindent\begin{proof}
Let
    \begin{equation}\label{F0.13}\begin{split}
        p_E &\colon E\times _MF \rightarrow E\,,  \\
        p_F & \colon E\times _M F\rightarrow F\,,
    \end{split}\end{equation}
be the canonical projections. We consider the mapping
    $$ \wedge ^2 \sT p_E \times _{\wedge ^2\textsf{T} M}\wedge ^2\sT p_F
    \colon  \wedge ^2\sT (E\times _M F) \rightarrow \wedge ^2 \sT E \times
    _{\wedge ^2\textsf{T} M} \wedge ^2 \sT F.$$
Since projections $p_E$ and $p_F$ are vector bundle morphisms, we have
    \begin{equation}\label{F0.14}\begin{split}
        \wedge ^2 \sT p_E (\sV^2(E\times _MF)) &= \sV^2(E)\,,   \\
        \wedge ^2 \sT p_F (\sV^2(E\times _MF)) &= \sV^2(F)\,.
    \end{split}\end{equation}
Moreover, $\wedge ^2 \sT p_E \times _{\wedge ^2\textsf{T} M}\wedge ^2\sT p_F$ is a vector bundle morphism of vector bundles over $E\times _MF$.
It is obvious that the kernel of this morphism is contained in $\sV^2(E\times _M F)$. It follows then from (\ref{F0.14}) that
$\wedge ^2 \sT p_E \times _{\wedge^2\textsf{T} M}\wedge ^2\sT p_F$ projects to an isomorphism of
$\wedge ^2\sT(E\times _MF)\big/ \sV^2(E\times _MF)$
and  $\left(\wedge ^2\sT E \big/ \sV^2(E) \right) \times_{\wedge ^2\textsf{T} M} \left(\wedge ^2\sT F \big/ \sV^2(F) \right)$.
\end{proof}

In the following, we denote by $\bar\wedge ^2\sT E$ the quotient bundle  $\wedge ^2\sT E \big/ \sV^2(E)$.
The above proposition implies that $\bdt^{\,2}\zd$ defines a pairing between $\bar\wedge ^2\sT E$ and $\bar\wedge ^2\sT F$ as bundles over $\wedge^2\sT M$. Such pairing (if not degenerate) can be used to define a vector bundle structure in the fibration $\wedge ^2 \sT\zt \colon \wedge ^2 \sT E \rightarrow \wedge ^2
\sT M $. We shall consider two special cases.

\subsection{The special case: $F= E^\*\otimes _M\sT^\* M$}\label{h:6}
Let  $F= E^\*\otimes _M\sT^\* M$ and let
    \begin{equation}\label{F0.15}\begin{split}
        \zd_E &\colon E\times _M(E^\*\otimes _M\sT^\* M) \rightarrow  \sT^\* M \\
        &\colon (e,\zx\otimes \za) \mapsto \langle \zx, e\rangle \za.
    \end{split}\end{equation}
In local coordinates,
    $$ \zd_E = e^a f_{a\zm}\xd x^\zm$$
and
    \begin{equation}\label{F0.16}
        \dt^2 \zd_E = -e^a y^\zn_{a\zn} - f_{a\zn}y^{\zn a},
    \end{equation}
where $y^{\zn a}$ are coordinates in $\wedge ^2\sT E$ and
$y^\zn_{a\zm}$ are coordinates in $\wedge ^2\sT (E^\*\otimes _M\sT^\* M)$.
We see from this formula and Proposition~(\ref{C0.3}) that
$\dt^2\zd$ projects to a function
    $$ \bdt^{\,2}\zd_E \colon \left(\wedge ^2\sT E \big/ \sV^2(E) \right) \times_{\wedge ^2\textsf{T} M} \left(\wedge ^2\sT F \big/ \sV^2(F) \right) \rightarrow \R.$$
Hence, each element of $u\in \wedge ^2 \sT E$ defines the function
    \begin{equation}\label{F0.17}\begin{split}
        g_u &\colon (\wedge ^2\sT(\zp\otimes _M\ezp_M))^{-1}(w)\rightarrow \R,
        \ \ w =\wedge ^2\sT\zt (u)   \\
        &\colon  v \mapsto \dt^2 \zd_E (u,v).
    \end{split}\end{equation}
Let $u,u' \in \wedge ^2\sT E$ be such that $\wedge ^2\sT\zt (u) =
\wedge ^2\sT\zt (u')$. It follows from (\ref{F0.16}) that $g_u = g_{u'}$ if and
only if $\zt^{2}_E(u) = \zt^{2}_E(u') $ and $u-u' \in
\sV^2(E)$. Moreover, the set $\{g_u \colon \wedge ^2\sT\zt(u) =w\}$ is a
vector subspace in the vector space of all functions on $(\wedge
^2\sT(\zp\otimes _M\ezp_M))^{-1}(w)$.
We have proved the following.
\begin{theorem}\label{C0.4}
There is a canonical vector bundle structure on the fibration
    $$ \wedge ^2\sT\zt \colon  \wedge ^2 \sT E \big/ \sV^2(E)\rightarrow  \wedge ^2 \sT M.$$
\end{theorem}
\noindent The coordinates
    $$(x^\zm, e^a, {\dot x}^{\zl\zn}, y^{\zk b}, {\dot e}^{cd})$$
on $\wedge ^2\sT E$, introduced by the formula (\ref{F0.3}), induce coordinates
    \be\label{coo}(x^\zm, e^a, {\dot x}^{\zl\zn}, y^{\zk b})
    \ee
on $\bar\wedge ^2 \sT E$.
The canonical projection
    $$\wedge ^2\sT\zt \colon  \bar\wedge ^2 \sT E\rightarrow  \wedge ^2 \sT M$$
reads in coordinates
$$\wedge ^2\sT\zt(x^\zm, e^a, {\dot x}^{\zl\zn}, y^{\zk b})=(x^\zm, {\dot x}^{\zl\zn})\,,$$
so that $(e^a, y^{\zk b})$ are linear coordinates and the Euler vector field reads
$$\bar X_{\bar\wedge ^2 \textsf{T} E}=e^a\pa_{e^a}+ y^{\zk b}\pa_{y^{\zk b}}\,.$$
Since it commutes with the standard Euler vector field
$$X_{\bar\wedge ^2 \textsf{T} E}={\dot x}^{\zl\zn}\pa_{{\dot x}^{\zl\zn}}+ y^{\zk b}\pa_{y^{\zk b}}\,,$$ the two vector bundle structures on
$\bar\wedge ^2 \sT E$ are compatible, i.e. form a double vector bundle. The core of this double vector bundle is canonically isomorphic to $\sT M\otimes_M E$.

\subsection{The special case: ${E=\sT M, F=\wedge ^2\sT^\* M }$}
We denote by
$\zD^2_M$ the mapping $\zd_{\sT M}$ restricted to $\sT M \times _M \wedge
^2\sT^\* M$:
    \begin{equation}\label{F0.20}\begin{split}
        \zD^2_M &\colon \sT M\times _M \wedge ^2\sT^\* M \rightarrow  \sT^\* M \\
        & \colon   (v,\za)\mapsto \ix_ v\za,
    \end{split}\end{equation}
where $\ix_v\za (w) = \za(v,w)$.
In local coordinates,
    $$ \zD^2_M = p_{\zm\zn} {\dot x}^\zm \xd x^\zn$$
and
   \begin{equation}\label{F0.21}
        \dt^2\zD^2_M = -p_{\zm\zn}y^{\zm\zn} - {\dot x}^\zn
        y^\zm_{\zm\zn} = -\frac{1}{2} p_{\zm\zn}(y^{\zm\zn} -y^{\zn\zm}) - {\dot x}^\zn
        y^\zm_{\zm\zn}.
   \end{equation}
Thus, the function $\dt^2\zD^2_M$ represents a `pairing' between the fibrations
    $$ \wedge ^2\sT \ezt_M \colon  \wedge ^2\sT\,\sT M \rightarrow\wedge ^2 \sT M$$
and
    $$ \wedge ^2\sT \ezp_M^{2} \colon  \wedge ^2\sT\,\wedge^2\sT^\* M \rightarrow \wedge ^2 \sT M$$
which projects to an actual pairing of vector bundles
    $$ \bar\wedge ^2\sT \ezt_M \colon  \bar\wedge ^2\sT\,\sT M \rightarrow\wedge ^2 \sT M$$
and
    $$ \bar\wedge ^2\sT \ezp_M^{2} \colon  \bar\wedge ^2\sT\,\wedge^2\sT^\* M  \rightarrow \wedge ^2 \sT M.$$
Since $\zD^2_M$ is the mapping $\zd_{\textsf{T} M}$ restricted to $\sT M \times _M \wedge^2\sT^\* M$, the pairing $\bdt^{\,2}\zD^2_M$ is bilinear with respect to vector bundle structures defined by $\bdt^{\,2}\zd_{\textsf{T} M}$ and $\bdt^{\,2}\zd_{\wedge^2\sT^\s* M}$, respectively.

\section{The mapping `kappa' for bivectors}\label{s:6}
As it was stressed in the introduction, a fundamental role in analytical mechanics is played by the canonical
isomorphism of double vector bundles
\begin{equation} \label{F0.22}
    {\xymatrix @C-1mm{  & \sT \sT M \ar[ldd]_*{\sT\zt_M} \ar[rd]^*{\zt_{\ssT M}}
    \ar[rrr]^*{\zk_M} & & & \sT \sT M \ar[ldd]_(.3)*{\zt_{\ssT M }}  \ar[rd]^*{\sT\zt_M} & \cr
    & & \sT M \ar[ldd]_(.3)*{\zt_M}  \ar[rrr]^(.7)*{\text{id}} & & & \sT M \ar[ldd]_*{\zt_M} \cr
    \sT M  \ar[rd]^*{\zt_M} \ar[rrr]^(.7)*{\text{id} } & & &
    \sT M \ar[rd]^*{\zt_M}  & &  \cr
    & M \ar[rrr]^*{\text{id}} & & & M &}} \ \ \
\end{equation}
which encodes the Lie algebroid structure of $\sT M$ \cite{GG,GU1}.

Let us consider a sequence of mappings
    $$\xymatrix@C+5mm{\sT\sT M \times _{\ezt_{\textsf{T} M}} \sT\sT M \ar[r]^{\zk_M\times
    \zk_M} & \sT\sT M\times _{\sT\ezt_M} \sT\sT M \ar[r]^/+4mm/{\sT \wedge ^2} & \sT \wedge ^2\sT M}.$$
Since $\wedge ^2$ is bilinear and skew-symmetric,  also $\sT \wedge
^2$ is bilinear  with respect to the tangent vector bundle structures and
skew-symmetric. It follows that the composition $\sT \wedge ^2 \circ
(\ezk_M\times \ezk_M)$  is bilinear and skew-symmetric and consequently,
defines a mapping
    \begin{equation}\label{F0.23}
        \ezk_M^2 \colon  \wedge ^2\sT \sT M \longrightarrow \sT \wedge ^2\sT M
    \end{equation}
which is a vector bundle morphism with respect to the vector bundle
structures over $\sT M$ (see \cite{GU}). It follows from the construction that
\begin{enumerate}
    \item $\ezt_{\wedge ^2\textsf{T} M} \circ  \ezk_M^2 = \wedge ^2\sT \ezt_M$,
        i.e. $\ezk_M^2$ respects the fibrations over $\wedge ^2 \sT M$,
    \item the subbundle $\sV^2(\sT M)$ is contained in the kernel of $\ezk_M^2$.
\end{enumerate}
The first  is an immediate consequence of the fact that $\ezk_M$ intertwines projections $\zt_{\sT M}$ and $\sT\ezt_M$. In order to see that  $\sV^2(\sT M)$ is contained in the kernel of $\ezk_M^2$, let us take two vertical vectors $w_1,w_2\in \sV \sT M$ with $ \ezt_{\textsf{T} M}(w_1)= \ezt_{\textsf{T} M}(w_2)$. The  vectors $\zk_M(w_i)$ have the properties: $\ezt_{\textsf{T} M}(\zk_M(w_i))=0$ and $\sT\ezt_{M}(\zk_M(w_1))= \sT\ezt_{M}(\zk_M(w_2))$. It follows
that in a local trivialization of $\sT M$, the vector $w_i$ can be represented by a curve in $\sT M$ of the form $t\mapsto (x(t), tv_i)$. Consequently, $\ezk^2_M(w_1\wedge w_2)$ is represented by the curve $(x(t),t^2v_1\wedge v_2)$, which represents a zero vector. By definition, $\sV^2(\sT M)$ is spanned by simple bi-vectors, i.e. of the form  $w_1\wedge w_2$, where $w_1,w_2\in \sV \sT M$. It follows that $\sV^2(\sT M)$ is contained in the kernel of $\ezk_M^2$.

In local coordinates,
    \begin{equation} \label{F0.24}
        (x^\zm, {\dot x}^{\zk\zl}, {x'}^{\zn}, {\dot
        x'}{}^{\zr\zs})\circ \ezk_M^2 = (x^\zm, {\dot x}^{\zk\zl}, {\dot x}^\zn,
        (y^{\zr\zs} - y^{\zs\zr}))\,.
    \end{equation}

Now, we find a coordinate-free description of the kernel of $\ezk^2_M$ as a vector bundle morphism over $\sT M$. As a tangent bundle of a vector bundle, $\sT \wedge ^2\sT M$ is a double vector bundle  (\cite{KU}) and it follows that the zero section of the vector bundle $\sT \zt_M^{\wedge 2}
\colon \sT \wedge ^2\sT M \rightarrow \sT M$ lies in the kernel of the projection $\ezt_{\wedge ^2\textsf{T} M}$. Since $\ezk_M^2$ respects the fibrations over $\wedge ^2\sT M$, we have $\wedge ^2\sT \ezt_M (u) =0 \in \wedge ^2\sT M$ for $u\in \ker \ezk_M^2$. As we have already noticed, $\sV^2(\sT M) \subset \ker \ezk_M^2$. Thus, the kernel is completely characterized by its image in the vector bundle $\bar\wedge ^2 \sT \sT M$. It is a general fact for a double vector bundle \ref{F03.1} with the core bundle $C$ that $\zt^{-1}_r(0)$ is canonically isomorphic to $E_l\times_M C$. It follows that
 $$ \bar\wedge ^2 \sT \sT M \supset (\ezt_{\wedge ^2\textsf{T} M})^{-1}(0) = \sT M\times_M (\sT M\otimes _M \sT M)\,,$$
 and the canonical projection $\ezt^{\wedge 2}_{\textsf{T} M}$ corresponds to the projection on the first component. Similarly,
    $$(\ezt_{\wedge ^2\textsf{T} M})^{-1}(0) = \sT M\times_ M \wedge ^2\sT M\,,$$
and we get the induced mapping
    $$ \ezk^2_M \colon \sT M\times_M (\sT M\otimes _M \sT M) \rightarrow\sT M\times_ M \wedge ^2\sT M \,.$$
It follows from the local formulae that
    $$\ezk^2_M \colon (v, u\otimes w)\mapsto (v, u\wedge w)$$
and that the kernel of $\ezk^2_M$ consists of bivectors in $\sV^1(\sT M)$,
which represent symmetric elements in $\sT M\otimes _M \sT M$.
It is now easy task to verify that $\ezk_M^2$ induces the following commutative diagram of maps
    \begin{equation}\label{F0.25a}
        {\xymatrix @C-2mm{   & \wedge ^2\sT \sT M  \ar[ldd]_*{\wedge^2\sT\ezt_M} \ar[rd]^*{\zt_{\textsf{T} M}^{\wedge 2}}
        \ar[rrr]^*{\zk_M^2} & & & \sT\wedge ^2 \sT M \ar[ldd]_(.3)*{\zt_{\wedge ^2\textsf{T} M }} \ar[rd]^*{\sT\ezt_M^{\wedge 2}} & \cr
        & & \sT M \ar[ldd]_(.3)*{\zt_M} \ar[rrr]^(.7)*{\text{id}} & & & \sT M \ar[ldd]_*{\zt_M}\cr
        \wedge ^2 \sT M  \ar[rd]^*{\zt_M^{\wedge 2}} \ar[rrr]^(.7)*{\text{id}} & & & \wedge ^2\sT M \ar[rd]^*{\zt_M^{2}} & &  \\
        & M \ar[rrr]^*{\text{id}}  & & & M &}} \ \ \ \,,
\end{equation}
which projects onto a morphism of double vector bundles:
    \begin{equation}\label{F0.25}
        {\xymatrix @C-2mm{   & \bar\wedge ^2\sT \sT M  \ar[ldd]_*{\bar\wedge^2\sT\ezt_M} \ar[rd]^*{\bar\zt_{\textsf{T} M}^{\wedge 2}}
        \ar[rrr]^*{\bar\zk_M^2} & & & \sT\wedge ^2 \sT M \ar[ldd]_(.3)*{\zt_{\wedge ^2\textsf{T} M }} \ar[rd]^*{\sT\ezt_M^{\wedge 2}} & \cr
        & & \sT M \ar[ldd]_(.3)*{\zt_M} \ar[rrr]^(.7)*{\text{id}} & & & \sT M \ar[ldd]_*{\zt_M}\cr
        \wedge ^2 \sT M  \ar[rd]^*{\zt_M^{\wedge 2}} \ar[rrr]^(.7)*{\text{id}} & & & \wedge ^2\sT M \ar[rd]^*{\zt_M^{2}} & &  \\
        & M \ar[rrr]^*{\text{id}}  & & & M &}} \ \ \ .
\end{equation}

\subsection{Graded bundle approach}\label{s5.3}
The question of the bundle structures in $\wedge^2\sT E$ can be approached differently following the concept of a graded bundle \cite{GR2} which generalizes that of a vector bundle (see also \cite{Roy, Sev}). The role of the Euler vector field is played by the {\it weight vector field} which in homogeneous coordinates is a linear diagonal vector field with weights being positive integers. An important observation is that
graded bundles of degree 1 are just vector bundles, as weight vector fields are Euler vector field. There is also a nice interpretation of graded bundles in terms of smooth actions of the multiplicative monoid $(\R,\cdot)$ \cite{GR,GR2}.

It is obvious, that
$\ezt^2_E\colon\wedge^2\sT E\rightarrow E$ is a vector bundle with the Euler vector field
\begin{equation}\label{h:3}
    X_{\wedge^2\textsf{T} E}=\dot x^{\zm\zn}\frac{\partial}{\partial \dot x^{\zm\zn}}+ y^{\zm a}\frac{\partial}{\partial y^{\zm a}}+\dot e^{bc}\frac{\partial}{\partial \dot e^{bc}}.
\end{equation}
Another graded bundle structure is given by the weight vector field $\xd_{\sT}^2X_E$ which in coordinates reads
\begin{equation}\label{h:1}
    \xd_{\textsf{T}}^2X_E=e^a\frac{\partial}{\partial e^{a}}+y^{\zm d}\frac{\partial}{\partial y^{\zm d}}+2\dot e^{bc}\frac{\partial}{\partial \dot e^{bc}}\,,
\end{equation}
so that makes $\wedge^2\sT E$ into a graded bundle of degree 2.
The coordinates $(x^\zm, \dot x^{\zm\zn})$ are of degree $0$,  $(e^a, y^{\zm b})$ are of degree $1$, and $(\dot e^{ab})$ are of degree $2$. The reduction to coordinates of degree $\le 1$ gives the vector bundle $\bar\wedge^2\sT\ezt\colon\bar\wedge^2\sT E\to \wedge^2 \sT M$. More precisely,
the canonical inclusions of function algebras
\begin{equation}\label{h:5}
        \mathcal{A}_2\hookleftarrow \mathcal{A}_1\hookleftarrow\mathcal{A}_0=\mathcal{C}^\infty(\wedge^2\sT M)\,,
    \end{equation}
where $\mathcal{A}_k$ is the algebra of polynomials in homogeneous functions of degree $\le k$,
induce the sequence of bundle projections
\begin{equation}\label{h:4}
        \wedge^2\sT E\longrightarrow \bar\wedge^2\sT E\longrightarrow \wedge^2 \sT M,
    \end{equation}
where the first is an affine, and the second is a vector bundle projection. Elements od $\mathcal{A}_1$ were interpreted in section \ref{h:6} as sections of the quotient bundle $\bar\wedge^2\sT E\rightarrow \wedge^2\sT M$, since coordinates (\ref{coo}) are clearly those of degree $\le 1$.
The first projection coincides with the canonical projection from $\wedge^2\sT E$ to the quotient space $\bar\wedge^2\sT E=\wedge^2\sT E\slash \sV^2 E$.

Note that the vector field $\xd_{\textsf{T}}^2X_E$ is the "bi-tangent" lift of the vector field $X_E$ to the bundle of bivectors. Recall that the tangent (complete) lift of a vector field $Y$ on a manifold $M$ to the vector field $\xd _\sT Y$ on $\sT M$ can be defined as
$$\xd_\textsf{T} Y=\kappa_M\circ\sT X.$$
In the context of bivectors we can do the same using $\wedge^2\sT$ functor and $\kappa_M^2$ map. The vector field $\xd_\textsf{T}^2 Y$ is defined as
$$\xd_\textsf{T}^2Y=\kappa_M^2\circ\wedge^2\sT Y.$$
In coordinates $(x^\zm)$ in $M$ and $(x^\zm, \dot x^{\zm\zn})$ in $\wedge^2\sT M$ the lift of a vector field $Y^\zm\frac{\partial}{\partial x^\zm}$ reads
$$\xd_\textsf{T}^2X=Y^\zm\frac{\partial}{\partial x^\zm}+\frac12\left(\frac{\partial Y^\zn}{\partial x^\zl}\dot x^{\zl\zk}-\frac{\partial Y^\zk}{\partial x^\zs}\dot x^{\zs\zn}\right)\frac{\partial}{\partial \dot x^{\zn\zk}}.$$
Applying the above formula to the Euler vector field $X_E=e^a\frac{\partial}{\partial e^a}$ of the bundle $\zt$ we get precisely (\ref{h:1}). It is easy to check that the vector fields $X_{\wedge^2\ssT E}$ and $\xd_\textsf{T}^2X$ commute, therefore $\wedge^2\sT E$ is a double graded bundle
    \begin{equation}\label{h:2}
        {\xymatrix@R-1mm @C-1mm{ & \wedge^2\sT E \ar[ld]_*{\zt^{\wedge 2}_{E}} \ar[rd]^*{\wedge^2 {\tau_M}} & \cr
        \quad E \ar[rd]_*{\tau} & & \wedge^2\sT M \ar[ld]^*{\zt^{\wedge2}_M}  \cr & M  & }} .
    \end{equation}
The weight vector fields $X_{\wedge^2\ssT E}$ and $\xd_\textsf{T}^2X$ define a bi-gradation in which the coordinates $(x^\zm, e^a, \dot x^{\zm\zn}, y^{\zm a}, \dot e^{ab})$ are of the bi-degrees
$(0,0)$, $(0,1)$, $(1,0), $(1,1), and $(1,2)$, respectively.

      Moreover, since $X_{\wedge^2\textsf{T} E}$ is projectable on $\bar\wedge^2\sT E$, its projection together with the projection of $\xd_{\textsf{T}}^2X_E$ define the double vector bundle structure on $\bar\wedge^2\sT E$:
    \begin{equation}\label{h:7}
        {\xymatrix@R-1mm @C-1mm{ & \bar\wedge^2\sT E \ar[ld]_*{\bar\zt^{\wedge 2}_{E}} \ar[rd]^*{\bar\wedge^2 {\tau_M}} & \cr
        \quad E \ar[rd]_*{\tau} & & \wedge^2\sT M \ar[ld]^*{\zt^{\wedge2}_M}  \cr & M  & }} .
    \end{equation}
We just forget the coordinates of degree $(1,2)$.

\medskip
Quite similarly, if we start with a graded bundle $G$ of degree 2 with coordinates
$(x^\zm, y^a, z^w)$ of degrees, respectively, $0$, $1$, $2$ with respect to the weight vector field $X_G=y^a\pa_{y^a}+2z^w\pa_{z^w}$, then $\sT G$ is a double graded bundle with respect to the Euler vector field
$$X_{\textsf{T} G}=\dot x^\zm\pa_{\dot x^\zm}+\dot y^a\pa_{\dot y^a}+\dot z^w\pa_{\dot z^w}$$
and the tangent lift
$$\xd_\textsf{T} X_G=y^a\pa_{y^a}+2z^w\pa_{z^w}+\dot y^a\pa_{\dot y^a}+2\dot z^w\pa_{\dot z^w}\,.$$
Thus, coordinates $(x^\zm, y^a, z^w, \dot x^\zm, \dot y^a, \dot z^w)$ are of bi-degrees $(0,0)$, $(0,1)$, $(0,2)$, $(1,0)$, $(1,1)$, and $(1,2)$, respectively.

\begin{example} The coordinates $(x^\zm, {\dot x}^\zn, {\dot x}^{\zk\zl}, y^{\zr\zs}, z^{\zf\zs})$
on $\we^2\sT\sT M$ are, respectively, of bi-degrees  $(0,0)$, $(0,1)$, $(1,0)$, $(1,1)$, and $(1,2)$. The coordinates $(x^\zm, {\dot x}^{\zk\zl}, {x'}^{\zn}, {\dot x'}{}^{\zr\zy})$ of $\sT\we^2\sT M$ are of bi-degrees $(0,0)$, $(0,1)$, $(1,0)$, and $(1,1)$.
Hence, the smooth map $\zk^2_M\colon\we^2\sT\sT M\to\sT\we^2\sT M$ given in coordinates by (\ref{F0.24})
respects bi-degrees, so (\ref{F0.25a}) is a morphism of double graded bundles.
\end{example}

\section{The Lagrangian side of the triple}\label{s:7}
The introduced in Section~5.2 function
    $$ \xd_\textsf{T}^2\zD^2_M \colon  \wedge ^2\sT \wedge ^2 \sT^\* M \times_{\wedge ^2\sT M} \wedge ^2 \sT\sT M \rightarrow \R$$
is a degenerate pairing between fibrations
    $$ \wedge ^2\sT \ezp^{\wedge 2}_M \colon \wedge ^2\sT \wedge ^2\sT^\* M \rightarrow  \wedge ^2\sT M$$
and
    $$ \wedge ^2\sT \ezt_M \colon \wedge ^2\sT \sT M \rightarrow\wedge ^2\sT M.$$
We have also the canonical pairing
    $$ \langle \,,\,\rangle \colon \sT^\* \wedge ^2\sT M \times
    _{\wedge ^2\sT M} \sT \wedge ^2\sT M \rightarrow \R.$$
With this pairings we define a relation
    $$ \eza^2_M \colon \wedge ^2\sT \wedge ^2 \sT^\* M  \rightarrow   \sT^\* \wedge ^2\sT M$$
which is the dual relation to the {\it converse} (called sometimes also: {\it inverse} or {\it transpose}) of $\ezk_M^2$. More precisely,
$p\in \eza_M^2(u)$ if and only if
$$ \forall\ w\in \wedge ^2 \sT\sT M \left[
\wedge ^2\sT \ezt_M (w) = \ezp_{\wedge ^2 \sT M}(p)\ \Rightarrow\  x \langle p, \ezk_M^2(w)\rangle  = \dt^2\zD^2_M (u,w)\right]\,.$$
In local coordinates, with the following representations of $p,w,u$:
\bea\label{F0.27}
        \wedge ^2\sT \sT M\ni w &=& \frac{1}{2} {\dot x}^{\zm\zn} \frac{\partial}{\partial x^\zm}\wedge \frac{\partial }{\partial x^\zn} + y^{\zm \zn}\frac{\partial }{\partial x^\zm}\wedge \frac{\partial }{\partial {\dot x}^\zn} + \frac{1}{2} {z}^{\zm\zn} \frac{\partial }{\partial {\dot x}^\zm}\wedge\frac{\partial }{\partial {\dot x}^\zn}, \nn \\
        \sT^\* \wedge ^2\sT M \ni p &=& p_\zm \xd x^\zm  + \frac{1}{2}f_{\zm\zn} \xd {\dot x}^{\zm\zn} \,, \ \ \  f_{\zm\zn}=-f_{\zn\zm}\,,\\
        \wedge ^2\sT \wedge^2 \sT^\* M\ni u &=& \frac{1}{2} {\dot x}^{\zm\zn} \frac{\partial}{\partial x^\zm}\wedge \frac{\partial }{\partial x^\zn} + \frac{1}{2} y^{\zm}_{\zn\zl}\frac{\partial }{\partial x^\zm}\wedge \frac{\partial }{\partial p_{\zn\zl}} +\frac{1}{8} {\dot p}_{\zm\zn\zl\zk} \frac{\partial }{\partial p_{\zm\zn}}\wedge\frac{\partial }{\partial p_{\zl\zk}},\nn
\eea
we have
        $$\wedge ^2 \sT M \ni \ezt_{\wedge ^2\textsf{T} M}(w) =  \frac{1}{2} {\dot x}^{\zm\zn}\frac{\partial}{\partial x^\zm}\wedge \frac{\partial }{\partial x^\zn}\,,\quad
        \sT M \ni \ezt^{\wedge 2}_M(w)= {\dot x}^\zn \frac{\partial }{\partial x^\zn}\,,
$$
and the equation $\langle p, \ezk_M^2(w)\rangle  = \dt^2\zD^2_M (u,w)$, in view of (\ref{F0.24}), reads
    \begin{equation}\label{F0.28}
        p_\zr {\dot x}^\zr + \frac{1}{2}(y^{\zl\zk} -
        y^{\zk\zl}) f_{\zl\zk} = -\frac{1}{2}p_{\zl\zk} (y^{\zl\zk} -
        y^{\zk\zl}) - y^\eta_{\eta\zr}{\dot x}^\zr .
    \end{equation}
It follows that the relation $\eza^2_M$ is given by
    \begin{equation}\label{F0.29}\begin{split}
        p_\zr&= -y_{\eta\zr}^\eta, \\
        f_{\zl\zk}      &= - p_{\zl\zk}
    \end{split}\end{equation}
and that it can be considered a mapping
    $$ \eza_M^2 \colon \wedge ^2\sT \wedge^2 \sT^\* M\longrightarrow \sT^\* \wedge ^2\sT M,$$
which projects to a mapping
    $$ \bar\eza_M^2 \colon \bar\wedge ^2\sT \wedge^2 \sT^\* M   \longrightarrow \sT^\* \wedge ^2\sT M,$$
which is a morphism of double vector bundles (but not an isomorphism). In other words,
\begin{equation}\label{F0.31a}
        (x^\zm, {\dot x}^{\zn\zs}, p_\zr,f_{\zl\zk})\circ \eza^2_M =
        (x^\zm, {\dot x}^{\zn\zs}, -y_{\eta\zr}^\eta, -p_{\zl\zk})\,,
    \end{equation}
where $(x^\zm, p_{\zl\zk},{\dot x}^{\zn\zs},y_{\theta\zr}^\eta,\dot p_{\gamma\delta\epsilon\zeta})$ are coordinates on  $\we^2\sT\we^2\sT^\ast M$.

\section{The Hamiltonian side of the triple}\label{s:8}
Let $\ezvy_M^2$ be the canonical Liouville 2-form on $\wedge ^2 \sT^\*M$,
    $$\ezvy^2_M = \frac{1}{2} p_{\zm\zn} \xd x^\zm\wedge \xd x^\zn.$$
Its exterior differential
    $$\ezw^2_M = \xd \ezvy_M^2 =   \frac{1}{2} \xd p_{\zm\zn} \wedge\xd x^\zm\wedge \xd x^\zn$$
defines the  canonical multisymplectic structure on $\wedge ^2 \sT^\*
M$. The 3-form $\ezw^2_M$ can be represented by a vector bundle morphism
    \begin{equation}\label{F0.30}\begin{split}
        \ezb^2_M &\colon  \wedge ^2\sT \wedge^2 \sT^\* M \rightarrow \sT^\*
        \wedge ^2 \sT^\* M  \\
        &\colon u \mapsto \xi_u \ezw^2_M.
    \end{split}\end{equation}
In local coordinates,
    \begin{equation}\label{F0.31}
        (x^\zm, p_{\zl\zk}, f_\zr, q^{\zn\zs})\circ \ezb^2_M =
        (x^\zm, p_{\zl\zk}, -y_{\eta\zr}^\eta, {\dot x}^{\zn\zs}),
    \end{equation}
where $f_\zr \xd x^\zr + \frac{1}{2} q^{\zn\zs} \xd p_{\zn\zs} $ is the
representation of elements in $\sT^\*\wedge ^2 \sT^\* M$ and $(x^\zm, p_{\zl\zk},{\dot x}^{\zn\zs},y_{\theta\zr}^\eta,\dot p_{\gamma\delta\epsilon\zeta})$ are coordinates on  $\we^2\sT\we^2\sT^\ast M$.

It follows that $\ezb^2_M$ projects  to a double vector bundle morphism
    $$ \bar\ezb^2_M \colon  \bar\wedge ^2\sT \wedge^2 \sT^\* M \rightarrow \sT^\*\wedge ^2 \sT^\* M  \,.$$
Moreover, like in the case of Tulczyjew triples for field theories \cite{G1,GG2}, we consequently view  $\eza^2_M$ and $\ezb^2_M$ as relations, thus writing $(\eza^2_M)^{-1}$ for the inverse relation of $\eza^2_M$, we get
    $$(x^\zm, p_{\zl\zk}, f_\zr, q^{\zn\zs})\circ \ezb^2_M\circ(\eza^2_M)^{-1} = (x^\zm, -f_{\zl\zk}, p_\zr, {\dot x}^{\zn\zs})\,,$$
i.e.
$$\ezb^2_M\circ (\eza^2_M)^{-1}\colon \sT^\*\wedge ^2 \sT M\to\sT^\*\wedge ^2 \sT^\*
M$$
is the canonical symplectomorphism, a particular case of the general isomorphism $\sT^\*E\simeq\sT^\*E^\*$ which holds for an arbitrary vector bundle $E$  (cf. \cite{KT,KU}). It is generated by the canonical pairing between $\we^2\sT^\* M$ and $\we^2\sT M$, i.e. by the function $\frac12 p_{\zm\zn}\dot x^{\zm\zn}$.  Another consequence of this fact is that the level sets of $\eza^2_M$ and $\ezb^2_M$ coincide.

\section{The Tulczyjew triple}\label{s:9}
With the use of the  mappings $\ezb^2_M$ and $\eza_M^2$, we
can transport canonical structures from $\sT^\*\wedge ^2\sT M$ and
$\sT^\*\wedge ^2\sT^\* M$ to $\wedge ^2\sT \wedge^2 \sT^\* M$. On the other hand,   we can  lift canonical
forms on $\wedge ^2\sT^\* M$ to forms on $\wedge ^2\sT \wedge^2 \sT^\* M $
by the operator $\dt^2$.

For the Liouville 2-form
$$\ezvy^2_M = \frac{1}{2} p_{\zm\zn} \xd x^\zm\wedge \xd x^\zn,   $$
we get
    \begin{equation}\label{F0.32}\begin{split}
        \xd \ezvy^2_M &=  \frac{1}{2} \xd p_{\zm\zn}\wedge  \xd x^\zm\wedge \xd x^\zn, \\
        \xit^2 \ezvy^2_M &= \frac{1}{2} p_{\zm\zn} {\dot x}^{\zm\zn}, \\
        \xit^2 \xd \ezvy^2_M    &=  \frac{1}{2} {\dot x}^{\zm\zn} \xd p_{\zm\zn}  - y^\zs_{\zs\zn} \xd x^\zn,\\
        \xd \xit^2  \ezvy^2_M   &=  \frac{1}{2} {\dot x}^{\zm\zn} \xd p_{\zm\zn}  + p_{\zm\zn} \xd {\dot x}^{\zm\zn},\\
        \xd_\textsf{T}^2 \ezvy^2_M     &=  \frac{1}{2} p_{\zm\zn} \xd {\dot x}^{\zm\zn} + y^\zs_{\zs\zn} \xd x^\zn\,.
    \end{split}\end{equation}
Let $\ezvy_{\wedge ^2\sT M}$ and $\ezvy_{\wedge ^2\sT^\s* M}$ be the canonical Liouville 1-forms on $\sT^\* \wedge ^2\sT M$ and  $\sT^\* \wedge^2\sT^\* M $, respectively. From  (\ref{F0.29}), (\ref{F0.31}), and (\ref{F0.32}) we derive
    \begin{equation}\label{F0.33}\begin{split}
        (\eza^2_M)^\* \ezvy_{\wedge ^2\textsf{T} M}     &= (\eza^2_M)^\*(p_\zm\xd x^\zm + \frac{1}{2} f_{\zm\zn} \xd {\dot x}^{\zm\zn} )  \\
                &= - \frac{1}{2} p_{\zm\zn} \xd {\dot x}^{\zm\zn} - y^\zs_{\zs\zn} \xd x^\zn\\
                &= - \dt^2 \ezvy^2_M
    \end{split}\end{equation}
and
    \begin{equation}\label{F0.34}\begin{split}
        (\ezb^2_M)^\* \ezvy_{\wedge ^2\textsf{T}^\s* M} &= (\ezb^2_M)^\*(f_\zm\xd x^\zm +\frac{1}{2}q^{\zm\zn} \xd p_{\zm\zn} )  \\
                &=  \frac{1}{2} {\dot x}^{\zm\zn} \xd p_{\zm\zn}- y^\zs_{\zs\zn} \xd x^\zn\\
                &=  \xit^2 \xd \ezvy^2_M = \xit^2\ezw^2_M.
    \end{split}\end{equation}
The above formulae imply immediately the following theorem, perfectly corresponding to analogous facts in the analytical mechanics of point-like objects.
\begin{theorem}\label{C0.5}
The mappings $\eza^2_M$ and $\ezb^2_M$ have the following property
    \begin{equation}\label{F0.35}\begin{split}
        (\eza^2_M)^\* \ezw_{\wedge ^2\textsf{T} M}&= \dt^2  \ezw^2_{M}\,, \\
        (\ezb^2_M)^\* \ezw_{\wedge ^2\textsf{T}^\s* M} &= \dt^2  \ezw^2_{M}\,.
    \end{split}\end{equation}
\end{theorem}

\begin{proof}
We have already proved (\ref{dd}) that $\xd$ (super)commutes with  $\dt^2$, i.e.
    $ \xd\dt^2 = -\dt^2\xd$\,.
Hence,
    $$ \dt^2\ezw^2_M = \dt^2\xd \ezvy^2_M = - \xd\dt^2\ezvy^2_M =(\eza^2_M)^\* \ezw_{\wedge ^2\textsf{T} M}$$
and
    $$ \dt^2\ezw^2_M = \xd \xit^2 \ezw^2_M = \xd\,(\ezb^2_M)^\*
    \ezvy_{\wedge ^2\sT^\s* M} = (\ezb^2_M)^\* \ezw_{\wedge ^2\textsf{T}^\s* M}\,.$$
\end{proof}

The second equality is a special case of (\ref{v}).  The form $\dt^2\ezw^2_M$ is a presymplectic 2-form (it is closed) (\cite{LM}). The degeneracy submanifolds of $\dt^2\ezw^2_M$ are just level sets of  $\eza^2_M,\ \ezb^2_M$. Hence, it the pull-back of a presymplectic form  on $\bar\wedge^2 \sT \wedge^2\sT^\* M$ to $\wedge ^2\sT \wedge^2 \sT^\* M $. In this way we have obtained canonical examples of presymplectic manifolds which are not symplectic.

Combining the maps $\ezb^2_M$ and $\eza_M^2$, we get the following {\it Tulczyjew triple}, being an analogue of the diagram (\ref{F3.1}):
    \be\label{TT}{\xymatrix@R-1mm @C-12mm{ &  \sT^\*\wedge^2\sT^\*M  \ar[ldd]_*{} \ar[rd]^*{} & & & \wedge^2 \sT \wedge^2\sT^\* M \ar[rrr]^*{{\za}^2_M}
    \ar[lll]_*{\zb^2_M} \ar[ldd]^*{} \ar[rd]^*{}& & & \sT^\*\wedge^2\sT M \ar[ldd]^*{} \ar[rd]^*{} & \cr
    & & \wedge^2\sT M \ar[ldd]^*{} & & & \wedge^2\sT M  \ar[ldd]^*{} \ar[lll]^*{} \ar[rrr]^*{} & & & \wedge^2\sT M \ar[ldd]^*{}  \cr
    \wedge^2\sT^\* M  \ar[rd]^*{}  & & & \wedge^2\sT^\* M \ar[rrr]^*{} \ar[lll]^*{} \ar[rd]^*{}& & & \wedge^2\sT^\* M
    \ar[rd]^*{} & & \cr
    & M  & & &  M \ar[rrr]^*{} \ar[lll]^*{} & & & M & }}\,.
    \ee
The above diagram consists of double graded bundle morphisms which simultaneously preserve the presymplectic structures.
Its reduction to degree 1 gives a diagram of morphisms of double vector bundles
$${\xymatrix@R-1mm @C-12mm{ &  \sT^\*\wedge^2\sT^\*M  \ar[ldd]_*{} \ar[rd]^*{} & & & \bar\wedge^2 \sT \wedge^2\sT^\* M \ar[rrr]^*{{\bar\za}^2_M}
    \ar[lll]_*{\bar\zb^2_M} \ar[ldd]^*{} \ar[rd]^*{}& & & \sT^\*\wedge^2\sT M \ar[ldd]^*{} \ar[rd]^*{} & \cr
    & & \wedge^2\sT M \ar[ldd]^*{} & & & \wedge^2\sT M  \ar[ldd]^*{} \ar[lll]^*{} \ar[rrr]^*{} & & & \wedge^2\sT M \ar[ldd]^*{}  \cr
    \wedge^2\sT^\* M  \ar[rd]^*{}  & & & \wedge^2\sT^\* M \ar[rrr]^*{} \ar[lll]^*{} \ar[rd]^*{}& & & \wedge^2\sT^\* M
    \ar[rd]^*{} & & \cr
    & M  & & &  M \ar[rrr]^*{} \ar[lll]^*{} & & & M & }}\,.
    $$

\section{Phase dynamics}\label{s:10}
\subsection{Lagrangian dynamics and Euler-Lagrange equations}
The way of obtaining the implicit phase dynamics $D$, as a submanifold of $\wedge^2 \sT \wedge^2\sT^\* M$, from a Lagrangian $L:\wedge^2\sT M\to\R$ is now standard: we take the lagrangian submanifold $\cL(L)=(\xd L)(\we^2\sT M)$ of the cotangent bundle $\sT^\*\we^2\sT M$ and then its inverse image $D=D(L)=(\za^2_M)^{-1}(\cL(L))$ under the map $\za^2_M$:

{$$\hskip-1cm
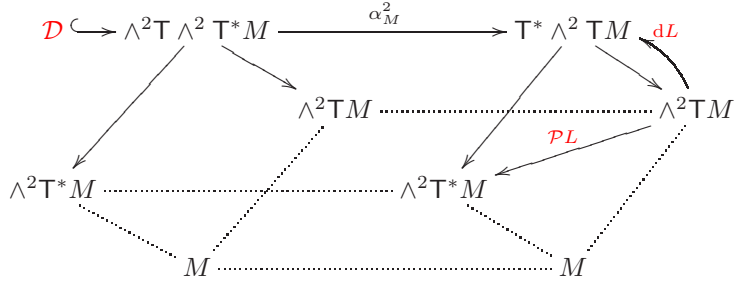
\xymatrix@C-20pt@R-8pt{
{\color{red} \mathcal{D}}\ar@{ (->}[r]& \we^2\sT\we^2\sT^\ast M \ar[rrr]^{\alpha_M^2} \ar[dr] \ar[ddl]
 & & & \sT^\ast\we^2\sT M\ar[dr]\ar[ddl] & \\
 & & \we^2\sT M\ar@{.}[rrr]\ar@{.}[ddl]
 & & & \we^2\sT M \ar@{.}[ddl]\ar@/_1pc/[ul]_{\color{red}\xd L}\ar[dll]_{\color{red}\mathcal{P} L}\\
 \we^2\sT^\ast M\ar@{.}[rrr]\ar@{.}[dr]
 & & & \we^2\sT^\ast M\ar@{.}[dr] & &  \\
 & M\ar@{.}[rrr]& & & M &
}\qquad$$}

\medskip\noindent
The phase dynamics
$$\mathcal{D}=(\alpha_M^2)^{-1}(\xd L(\we^2\sT M)))\,,$$
read in coordinates
$$\mathcal{D}=\left\{(x^\zm, p_{\zl\zk},{\dot x}^{\zn\zs},y_{\theta\zr}^\eta,\dot p_{\gamma\delta\epsilon\zeta}):\;\; y_{\eta\zr}^\eta=-\frac{\partial L}{\partial x^\zr},\quad p_{\zl\zk}=-\frac{\partial L}{\partial \dot{x}^{\zl\zk}}\right\}\,,$$
so the Lagrange (phase) equations are
\beas \frac{\partial L}{\partial x^\zr}&=&-y_{\eta\zr}^\eta\,,\\
\frac{\partial L}{\partial \dot{x}^{\zl\zk}}&=&-p_{\zl\zk}\,.
\eeas
The map
$$\mathcal{P}L:\we^2\sT M\to\we^2\sT^\ast M\,,\quad (x^\zm,{\dot x}^{\zn\zs})\mapsto\left(x^\zm,-\frac{\partial L}{\partial \dot{x}^{\zn\zs}}\right) $$
is the \emph{Legendre map}.

A surface $S:(t,s)\mapsto (x^\zs(t,s))$ in $M$ satisfies the Euler-Lagrange equations if the image by $\xd L$  of its prolongation to $\we^2\sT M$,
$$(t,s)\mapsto \left(x^\zs(t,s)),\dot x^{\zm\zn}=\frac{\pa x^\zm}{\pa t}\frac{\pa x^\zn}{\pa s}-\frac{\pa x^\zm}{\pa s}\frac{\pa x^\zn}{\pa t}\right)\,,$$
is $\za_M^2$-related to the prolongation of the surface $\mathcal{P}L\circ \we^2\sT S$ to the phase space.
In coordinates, the Euler-Lagrange equations read
\bea\label{el1}\dot x^{\zm\zn}&=&\frac{\pa x^\zm}{\pa t}\frac{\pa x^\zn}{\pa s}-\frac{\pa x^\zm}{\pa s}\frac{\pa x^\zn}{\pa t}\,,\\
\frac{\pa L}{\pa x^\zs}&=&\frac{\pa x^\zm}{\pa t}\frac{\pa}{\pa s}\left(\frac{\pa L}{\pa \dot x^{\zm\zs}}(t,s)\right)-\frac{\pa x^\zm}{\pa s}\frac{\pa}{\pa t}\left(\frac{\pa L}{\pa \dot x^{\zm\zs}}(t,s)\right)\,.\label{ee2}
\eea

\begin{remark} Let us note at this point that some signs in our approach are the matter of several conventions. For instance, one can write  $\dt^2$ as the commutator $\xit^2 \xd-\xd\, \xit^2 $, opposite to what we have used. Similarly, one sometimes uses the convention in which the contractions with multivector fields satisfies the condition $i_{X\wedge Y}=i_X\circ i_Y$. Of course, in this case the pairing between $\pa_{x^\zm}\we\pa_{x^\zn}$ and $\xd {x^\zm}\we\xd {x^\zn}$ gives $-1$, so one gets coordinates in which $p_{\zm\zn}$ and $y_{\eta\zr}^\eta$ differ by sign from ours. We get, however, the Lagrange equations of the form formally closer to the standard one:
\beas
\frac{\partial L}{\partial x^\zr}&=&y_{\eta\zr}^\eta\,,\\
\frac{\partial L}{\partial \dot{x}^{\zl\zk}}&=&p_{\zl\zk}\,.
\eeas
\end{remark}

\subsection{Hamilton equations}
Similarly, for a Hamiltonian $H:\wedge^2\sT^\* M\to\R$, we take the inverse image $D=D(H)=(\zb^2_M)^{-1}(\cL(H))$ of the lagrangian submanifold $\cL(H)=\xd H(\we^2\sT^\* M)$ of the cotangent bundle $\sT^\*\we^2\sT^\* M$:
{$$\hskip-1.2cm
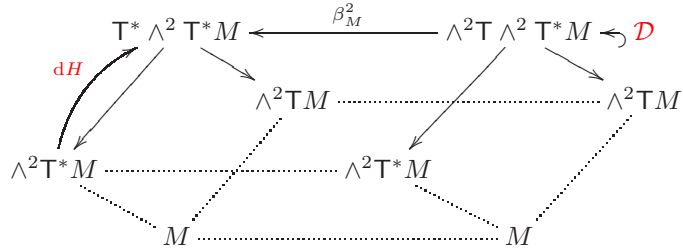
\xymatrix@C-25pt@R-12pt{
 & \sT^\ast\we^2\sT^\ast M  \ar[dr] \ar[ddl]
 & & & \we^2\sT\we^2\sT^\ast M\ar[dr]\ar[ddl] \ar[lll]_{\beta_M^2}&
 {\color{red} \mathcal{D}}\ar@{ (->}[l] \\
 & & \we^2\sT M\ar@{.}[rrr]\ar@{.}[ddl]
 & & & \we^2\sT M \ar@{.}[ddl]\\
 \we^2\sT^\ast M\ar@{.}[rrr]\ar@{.}[dr] \ar@/^1pc/[uur]^{\color{red}\xd H}
 & & & \we^2\sT^\ast M\ar@{.}[dr] & &  \\
 & M\ar@{.}[rrr]& & & M &
}\qquad$$}
The phase dynamics
$$\mathcal{D}=(\beta_M^2)^{-1}(\xd H(\we^2\sT^\ast M))$$
read in coordinates
$$\mathcal{D}=\left\{(x^\zm, p_{\zl\zk},{\dot x}^{\zn\zs},y_{\theta\zr}^\eta,\dot p_{\gamma\delta\epsilon\zeta}):\;\; y_{\eta\zr}^\eta=-\frac{\partial H}{\partial x^\zr},\quad {\dot x}^{\zn\zs}=\frac{\partial H}{\partial p_{\zn\zs}}\right\}\,,$$
so the Hamilton equations are
\beas\frac{\pa H}{\pa p_{\zm\zn}}&=&\frac{\pa x^\zm}{\pa t}\frac{\pa x^\zn}{\pa s}-\frac{\pa x^\zm}{\pa s}\frac{\pa x^\zn}{\pa t}\,,\\
-\frac{\pa H}{\pa x^\zs}&=&\frac{\pa x^\zm}{\pa t}\frac{\pa p_{\zm\zs}}{\pa s}-\frac{\pa x^\zm}{\pa s}\frac{\pa p_{\zm\zs}}{\pa t}\,.
\eeas
Note that the integrable part of the dynamics $D$ is contained in the set of decomposable (simple) bivectors.

This can be extended to the framework of Morse families replacing the single Hamiltonian. Let $p\colon B \rightarrow N$ be a submersion of a differentiable manifold $B$ onto a differentiable manifold $N$. Recall that a differentiable function $H\colon B\rightarrow \mathbb{R}$ is called a \textit{Morse family} if the image $\cL(H)=\xd H(B)\subset \sT^\*B$ of the differential of $H$ is transversal to the conormal bundle $V=(\ker \sT p)^0\subset \sT^\*B$. Note that in this situation there is a canonical vector bundle morphism $p^*:V\to\sT^\*N$ which maps $V\cap\xd H(B)$ onto a lagrangian submanifold $D(H)$ of $(\sT^\* M,\zw_N)$ which is said {\it to be generated by the Morse family $H$}.

\section{An example}\label{s:11}
In the dynamics of strings, the manifold of infinitesimal
configurations  is $\wedge ^2 \sT M$, where $M$ is the space time with the
Lorentz metric $g$. This metric induces a scalar product $h$ in fibers of
$\wedge ^2 \sT M$: for
    $$ w=\frac{1}{2} {\dot x}^{\zm\zn}\frac{\partial }{\partial
    x^\zm}\wedge \frac{\partial }{\partial x^\zn},\ \
    u=\frac{1}{2} {\dot x'}{}^{\zm\zn}\frac{\partial }{\partial
    x^\zm}\wedge \frac{\partial }{\partial x^\zn}
    $$
we have
    \begin{equation}\label{F1}
        (u|w) = h_{\zm\zn\zk\zl} {\dot x}^{\zm\zn}{\dot x'}{}^{\zk\zl},
    \end{equation}
where
    $$h_{\zm\zn\zk\zl} =\frac{1}{2}\left( g_{\zm\zk} g_{\zn\zl} - g_{\zm\zl}g_{\zn\zk}\right) .$$

The Lagrangian is a function of the volume with respect to this metric, the so called {\it Nambu-Goto Lagrangian} \cite[Chapter 2.2]{LT} (cf. also \cite[Example 4.2]{S}),
    \begin{equation}\label{F2}
        L(w) =\sqrt{(w|w)} = \sqrt{h_{\zm\zn\zk\zl} {\dot x}^{\zm\zn}{\dot x}{}^{\zk\zl}}\,,
    \end{equation}
which is defined on the open submanifold of positive bivectors.

The dynamics
    $$ D \subset \wedge ^2\sT \wedge^2 \sT^\* M$$
is the inverse image by $\eza^2_M$ of the image of $\xd L$ and it is
described by equations
    \begin{equation}\label{F3}\begin{split}
        y^\za_{\za\zn}&= - \frac{1}{2\zr} \frac{\partial
        h_{\zm\zk\zl\zs}}{\partial x^\zn} {\dot x}^{\zm\zk} {\dot x}^{\zl\zs},   \\
                       p_{\zm\zn} &= - \frac1\zr h_{\zm\zn\zl\zk}{\dot x}^{\zl\zk}\,,                      \end{split}\end{equation}
where
$$\zr = \sqrt{h_{\zm\zn\zl\zk}{\dot x}^{\zm\zn}{\dot x}^{\zl\zk}}$$
is the pull-back of our Lagrangian.
The dynamics $D$ is also the inverse image by $\ezb^2_M$ of the lagrangian submanifold in $\sT^\*\wedge^2\sT^\* M$, generated  by the Morse family
    \begin{equation}\label{F4}\begin{split}
        H&\colon \wedge^2\sT^\* M\times \R_+ \rightarrow \R \\
        &\colon (p,r)\mapsto r(\sqrt{(p|p)} -1)
    \end{split}\end{equation}
In the case of minimal surface, i.e. {\it the Plateau problem}, (see e.g. \cite{E}), we replace the Lorentz metric by a positively defined one.

In particular, if $M=\R^3=\{(x^1=x,x^2=y,x^3=z)\}$ with the Euclidean metric, the Lagrangian reads
$$L(x^\zm,\dot x^{\zk\zl})=\sqrt{\sum_{\zk,\zl=1}^3\left(\dot x^{\zk\zl}\right)^2}
$$
and we obtain the Lagrange (phase) equations
\beas
        & &y^\za_{\za\zn}= 0\,,   \\
        & & p_{\zm\zn} = -\frac1\zr\zd_{\zk\zm}\,\zd_{\zl\zn}\,{\dot x}^{\zk\zl}\,.
\eeas
The Euler-Lagrange equations (\ref{el1})-(\ref{ee2}), applied for surfaces being graphs $(x,y)\mapsto (x,y,z(x,y))$, in this case read
\be\label{e0}\dot x^{12}=1\,,\quad \dot x^{13}=\frac{\pa z}{\pa y}=z_y\,,\quad \dot x^{23}=-\frac{\pa z}{\pa x}=-z_x\,,\ee
$$(E_k):\quad\frac{\pa}{\pa y}\left(\frac{\dot x^{1k}}{\zr}\right)+z_x\frac{\pa}{\pa y}\left(\frac{\dot x^{3k}}{\zr}\right)-\frac{\pa}{\pa x}\left(\frac{\dot x^{2k}}{\zr}\right)-z_y\frac{\pa}{\pa x}\left(\frac{\dot x^{3k}}{\zr}\right)=0\,,\quad k=1,2,3\,.$$
In view of (\ref{e0}), the equation $(E_3)$ can be rewritten as
\be\label{e3}
\frac{\pa}{\pa y}\left(\frac{z_y}{\zr}\right)+\frac{\pa}{\pa x}\left(\frac{z_x}{\zr}\right)=0\,.
\ee
But (\ref{e0}) and (\ref{e3}) imply the rest. Indeed, $(E_1)$ can be rewritten as
$$-z_x\frac{\pa}{\pa y}\left(\frac{z_y}{\zr}\right)+\frac{\pa}{\pa x}\left(\frac{1}{\zr}\right)+
z_y\frac{\pa}{\pa x}\left(\frac{z_y}{\zr}\right)=0\,.
$$
With the presence of (\ref{e3}), this is equivalent to
$$z_x\frac{\pa}{\pa x}\left(\frac{z_x}{\zr}\right)+\frac{\pa}{\pa x}\left(\frac{1}{\zr}\right)+
z_y\frac{\pa}{\pa x}\left(\frac{z_y}{\zr}\right)=0
$$
and further to
$$\frac{\pa}{\pa x}\left(\frac{z_x^2+1+z_y^2}{\zr}\right)-\frac{z_xz_{xx}+z_yz_{xy}}{\zr}=0\,.
$$
But the latter is tautological, as the first summand is
$$\frac{\pa \zr}{\pa x}=\frac{\pa}{\pa x}\left(\sqrt{z_x^2+1+z_y^2}\right)=\frac{z_xz_{xx}+z_yz_{xy}}{\zr}\,.
$$
The equation $(E_2)$ follows in a similar way.
Thus, we have finished with (\ref{e3}) which is the well-known equation for minimal surfaces, known already to Lagrange:
$$\frac{\pa}{\pa x}\left(\frac{z_x}{\sqrt{1+z_x^2+z_y^2}}\right)+\frac{\pa}{\pa y}\left(\frac{z_y}{\sqrt{1+z_x^2+z_y^2}}\right)=0\,.
$$
In another form:
$$(1+z_x^2)z_{yy}-2z_xz_yz_{xy}+(1+z_y^2)z_{xx}=0\,.
$$

\section{Generalization}\label{s:12}
The diagram (\ref{TT}) has a straightforward generalization for all integer $n\ge 1$ replacing 2:
\be\label{TT1}{\xymatrix@R-1mm @C-12mm{ &  \sT^\*\wedge^n\sT^\*M  \ar[ldd]_*{} \ar[rd]^*{} & & & \wedge^n \sT \wedge^n\sT^\* M \ar[rrr]^*{{\za}^n_M}
    \ar[lll]_*{\zb^n_M} \ar[ldd]^*{} \ar[rd]^*{}& & & \sT^\*\wedge^n\sT M \ar[ldd]^*{} \ar[rd]^*{} & \cr
    & & \wedge^n\sT M \ar[ldd]^*{} & & & \wedge^n\sT M  \ar[ldd]^*{} \ar[lll]^*{} \ar[rrr]^*{} & & & \wedge^n\sT M \ar[ldd]^*{}  \cr
    \wedge^n\sT^\* M  \ar[rd]^*{}  & & & \wedge^n\sT^\* M \ar[rrr]^*{} \ar[lll]^*{} \ar[rd]^*{}& & & \wedge^n\sT^\* M
    \ar[rd]^*{} & & \cr
    & M  & & &  M \ar[rrr]^*{} \ar[lll]^*{} & & & M & }}\,.
    \ee
The map $\ezb^n_M$ comes from the canonical multisymplectic $(n+1)$-form $\ezw^n_M$ on $\wedge ^n \sT^\* M$, being the differential of the canonical Liouville $n$-form $\ezvy_M^n$:
\be\begin{split}
        \ezb^n_M &\colon  \wedge ^n\sT \wedge^n \sT^\* M \rightarrow \sT^\*
        \wedge ^n \sT^\* M  \\
        &\colon u \mapsto \xi_u \ezw^n_M.
    \end{split}
    \ee
In the standard local coordinates,
    $$\ezvy^n_M = \frac{1}{n!} p_{\zm_1\cdots\zm_n} \xd x^{\zm_1}\wedge\cdots\wedge \xd x^{\zm_n}\,,$$
    and
    $$\ezw^n_M = \xd \ezvy_M^n =   \frac{1}{n!} \xd  p_{\zm_1\cdots\zm_n}\we \xd x^{\zm_1}\wedge\cdots\wedge \xd x^{\zm_n}\,.$$
The map ${\za^n_M}$ is just the composition of ${\zb^n_M}$ with the canonical isomorphism of double vector bundles $\sT^\*\wedge ^n \sT^\* M$ and $\sT^\*\wedge ^n \sT M$, but it can be obtained as the dual of the converse of the generalized `kappa' \cite{GU}:

\begin{equation}\label{kappa}
        {\xymatrix @C-2mm{   & \wedge ^n\sT \sT M  \ar[ldd]_*{\wedge^n\sT\ezt_M} \ar[rd]^*{\zt_{\sT M}^{\wedge n}}
        \ar[rrr]^*{\zk_M^n} & & & \sT\wedge ^n \sT M \ar[ldd]_(.3)*{\zt_{\wedge ^n\sT M }} \ar[rd]^*{\sT\ezt_M^{\wedge n}} & \cr
        & & \sT M \ar[ldd]_(.3)*{\zt_M} \ar[rrr]^(.7)*{\text{id}} & & & \sT M \ar[ldd]_*{\zt_M}\cr
        \wedge ^n \sT M  \ar[rd]^*{\zt_M^{\wedge n}} \ar[rrr]^(.7)*{\text{id}} & & & \wedge ^n\sT M \ar[rd]^*{\zt_M^{n}} & &  \\
        & M \ar[rrr]^*{\text{id}}  & & & M &}} \,.
\end{equation}

\medskip\noindent
The duality between $\wedge^n \sT \wedge^n\sT^\* M$ and $\wedge ^n\sT \sT M$ is represented by the function $\dt^n\zD^n_M$, where
 \begin{equation}\begin{split}
        \zD^n_M &\colon \sT M\times _M \wedge ^n\sT^\* M \rightarrow \wedge^{n-1} \sT^\* M \\
        & \colon   (v,\za)\mapsto \ix_ v\za\,,
    \end{split}\end{equation}
$\dt^n=[\xd\,,\xit^n]$, and
\beas
       & \xit^n\colon \zF^p(\ezp_M) \rightarrow  \zF^{p-n}(\ezp_{\wedge ^n\sT M})\,,\\
        & \xit^n\zf(v_1, v_2, v_3,\dots, v_{p-n}) = \zf(u, \sT \ezt^{\wedge n}_M(v_1), \dots,\sT \ezt^{\wedge n}_M(v_{p-n}))\,,
        \eeas
for $v_i\in \sT_u\wedge ^n\sT M$, $u\in \wedge ^n\sT M$.

The generation of dynamics from a Lagrangian or Hamiltonian is standard and the corresponding dynamics contain, as a particular example, minimal submanifolds (of dimension $n$) of $M$ (cf. \cite{X}).


\medskip
Received xxxx 20xx; revised xxxx 20xx.
\medskip


\begin{thebibliography}{99}

\bibitem{Bu}
\newblock C. Buttin,
\newblock \emph{Th\'eorie des op\'erateurs diff\'erentiels gradu\'es sur les formes diff\'erentielles},
\newblock Bull. Soc. Math. France {\bf 102} (1974),  49--73.

\bibitem{CIL}
\newblock F. Cantrijn, L.A. Ibort, M. De Leon,
\newblock \emph{Hamiltonian structures on multisymplectic manifolds},
\newblock {Rend. Sem. Mat. Univ. Pol. Torino}, {\bf 54} (1996), 225--236.

\bibitem{CGM}
\newblock C. M. Campos, E. Guzm\'{a}n, J.C. Marrero,
\newblock \emph{Classical field theories of first order and lagrangian
submanifolds of premultisymplectic manifolds},
\newblock  J. Geom. Mech.  {\bf 4}  (2012),  1--26.

\bibitem{CCI1}
\newblock J. F. Cari\~nena, M. Crampin, L. A. Ibort,
\newblock \emph{On the multisymplectic formalism for first order theories},
\newblock {Differential Geom. Appl.}, {\bf 1} (1991), 354--374.


\bibitem{EM}
\newblock A. Echeverr\'ia-Enr\'iquez, M.C. Mu\~noz-Lecanda,
\newblock \emph{Geometry of multisymplectic Hamiltonian
first order theory},
\newblock J. Math. Phys., {\bf 41} (2000), 7402--7444.

\bibitem{E}
\newblock L.~E.~Evans,
\newblock \emph{Partial differential equations},
\newblock Graduate Studies in Mathematics {\bf 19}, American Mathematical Society, Providence, RI, 1998.

\bibitem{FP1}
\newblock M. Forger, C. Paufler, H. R\"{o}mer,
\newblock \emph{A general construction of Poisson brackets on exact
multisymplectic manifolds},
\newblock Rep. Math. Phys., {\bf 51} (2003), 187--195.

\bibitem{FP2}
\newblock M. Forger, C. Paufler, H. R\"{o}mer,
\newblock \emph{Hamiltonian multivector fields and Poisson forms in multisymplectic field theories},
\newblock J. Math. Phys., {\bf 46}  (2005), 112903--112932.

\bibitem{FG}
\newblock M.~Forger and L.~G.~Gomes,
\newblock {Multisymplectic and polysymplectic structures on fiber bundles},
\newblock Rev. Math. Phys. {\bf 25}  (2013), 1350018 [47 pages].

\bibitem{Ga}
\newblock K. Gaw\c edzki,
\newblock \emph{On the geometrization of the canonical formalism in the classical field theory}, \newblock Rep. Math. Phys., {\bf 3} (1972), 307--326.

\bibitem{GM}
\newblock G. Giachetta, L. Mangiarotti,
\newblock \emph{Constrained Hamiltonian Systems and Gauge Theories},
\newblock {Int. J. Theor. Phys.}, {\bf 34} (1995), 2353--2371.

\bibitem{GMS}
\newblock G. Giachetta, L. Mangiarotti, G. A. Sardanashvili,
\newblock \emph{Advanced Classical Field Theory},
\newblock World Scientific, Singapore (2009).

\bibitem{GIMa}
\newblock M. J. Gotay, J. Isenberg, J. E. Marsden,
\newblock \emph{Momentum maps and classical relativistic fields, Part I: Covariant field theory},
\newblock preprint, arXiv: physics/9801019.

\bibitem{GIMb}
\newblock M. J. Gotay, J. Isenberg, J. E. Marsden,
\newblock \emph{Momentum maps and classical relativistic fields, Part II:
Canonical analysis of field theories},
\newblock preprint, arXiv: math-ph/0411032.

\bibitem{G1}
\newblock K. Grabowska,
\newblock {The Tulczyjew triple for classical Fields},
\newblock J. Phys. A {\bf 45} (2012), 145207--145242.

\bibitem{G2}
\newblock K. Grabowska,
\newblock \emph{Lagrangian and Hamiltonian formalism in field theory: a simple model},
\newblock J. Geom. Mech., {\bf 2} (2010), 375-395.

\bibitem{GG}
\newblock K.~Grabowska and J.~Grabowski,
\newblock \emph{Variational calculus with constraints on general algebroids},
\newblock J. Phys. A: Math. Theor. {\bf 41} (2008), 175204 (25pp).

\bibitem{GG1}
\newblock K.~Grabowska and J.~Grabowski,
\newblock \emph{Dirac Algebroids in Lagrangian and Hamiltonian Mechanics},
\newblock J. Geom. Phys. {\bf 61} (2011), 2233--2253.

\bibitem{GG2}
\newblock K.~Grabowska and J.~Grabowski,
\newblock \emph{Tulczyjew triples: from statics to field theory},
\newblock J. Geom. Mech. {\bf 5} (2013), 445--472.

\bibitem{GGU2}
\newblock K. Grabowska, J. Grabowski, P. Urba\'nski,
\newblock \emph{Geometrical Mechanics on algebroids},
\newblock Int. J. Geom. Meth. Mod. Phys., {\bf 3} (2006), 559--575. 

\bibitem{GR}
\newblock J. Grabowski and M.~Rotkiewicz,
\newblock \emph{Higher vector bundles and multi-graded symplectic manifolds},
\newblock J. Geom. Phys. {\bf 59} (2009),  1285--1305.

\bibitem{GR2}
\newblock J. Grabowski and M.~Rotkiewicz,
\newblock \emph{Graded bundles and homogeneity structures},
\newblock J.Geom. Phys. {\bf 62} (2011), 21--36.

\bibitem{GU}
\newblock J. Grabowski and P.~Urba\'nski,
\newblock \emph{Tangent lifts of Poisson and related structures},
\newblock J. Phys. A  {\bf 28} (1995), 6743--6777.


\bibitem{GU1}
\newblock J.~Grabowski and P.~Urba\'nski,
\newblock \emph{Algebroids -- general differential calculi on vector bundles},
\newblock J. Geom. Phys. {\bf 31} (1999), 111--1141.

\bibitem{Gun}
\newblock C.~G\"unther,
\newblock \emph{The polysymplectic Hamiltonian formalism in field theory and calculus of variations. I. The local case},
\newblock J. Differential Geom.  {\bf 25}  (1987), 23–-53.

\bibitem{KS}
\newblock J. Kijowski, W. Szczyrba,
\newblock \emph{A canonical structure for classical field theories},
\newblock {Commun. Math. Phys.}, {\bf 46} (1976), 183--206.

\bibitem{KTu}
\newblock J. Kijowski J., W. M. Tulczyjew,
\newblock \emph{A symplectic framework for field theories},
\newblock {Lecture Notes in Physics}, {\bf 107}, (1979).

\bibitem{KT}
\newblock I.~Kol\'a\v r and J.~Tom\'a\v s,
\newblock \emph{Gauge-natural transformations of some cotangent bundles},
\newblock {Acta Univ. M. Belii ser. Mathematics} {\bf 5} (1997), 3--15.

\bibitem{KU}
\newblock K. Konieczna, P. Urba\'nski,
\newblock \emph{Double vector bundles and duality},
\newblock Arch. Math. (Brno), {\bf 35}  (1999), 59--95.

\bibitem{Kr}
\newblock O. Krupkov\'{a},
\newblock Hamiltonian field theory,
\newblock \emph{J. Geom. Phys.} {\bf 43} (2002), 93--132.


\bibitem{LMS}
\newblock M. de Le\'on, D. Mart\'in de Diego, A. Santamar\'ia-Merino,
\newblock \emph{Tulczyjew's triples and lagrangian submanifolds in classical field theories},
\newblock in "Applied Differential Geometry and Mechanics" (eds. W. Sarlet and F. Cantrijn), Univ. of Gent, Gent, Academia Press, (2003), 21--47.

\bibitem{LM}
\newblock P.~Libermann and C.~M.~Marle,
\newblock \emph{Symplectic Geometry and Analytical Mechanics},
\newblock D.~Reidel Publishing Co. (1987).

\bibitem{LT}
\newblock D. L\"ust and S. Theisen,
\newblock \emph{Lectures on string theory},
\newblock Lecture Notes in Physics, {\bf 346}, Springer-Verlag, Berlin, 1989.

\bibitem{Lu}
\newblock M.~\L ukasik,
\newblock \emph{Rachunek wariacyjny niezale\.zny od parametryzacji. Przypadek jednowymiarowy} (Polish),
\newblock PhD Thesis, University of Warsaw, 2012.

\bibitem{Ma}
\newblock G.~Martin,
\newblock \emph{A Darboux theorem for multi-symplectic manifolds},
\newblock Lett. Math. Phys.  {\bf 16}  (1988), 133–-138.

\bibitem{PT}
\newblock G.~Pidello and W.~Tulczyjew,
\newblock \emph{Derivations of differential forms on jet bundles},
\newblock Ann. Mat. Pura Appl.  {\bf 147} (1987), 249--265.

\bibitem{Pr}
\newblock J.~Pradines,
\newblock \emph{Repr\'esentation des jets non holonomes par des morphismes vectoriels doubles soud\'es},
\newblock C. R. Acad. Sci. Paris, s\'erie A, {\bf 278} (1974), 1523--1526.

\bibitem{Roy}
\newblock D.~Roytenberg,
\newblock \emph{On the structure of graded symplectic
supermanifolds and Courant algebroids},
\newblock in {\sl Quantization, Poisson brackets and beyond (Manchester, 2001)},{Contemp. Math.} {\bf 315}, Amer. Math. Soc., Providence, RI, 2002, pp. 169--185.

\bibitem{S}
\newblock G.~Sardanashvily,
\newblock \emph{Lagrangian dynamics of submanifolds. Relativistic mechanics},
\newblock J. Geom. Mech.  {\bf 4}  (2012),  no. 1, 99--110.


\bibitem{Sev}
\newblock P.~{\v{S}}evera,
\newblock \emph{Some title containing the words "homotopy" and "symplectic", e.g. this one}, \newblock {Travaux math\'ematiques}, Univ. Luxemb., {\bf 16} (2005), 121--137.

\bibitem{Tu1}
\newblock W.~Tulczyjew,
\newblock \emph{Hamiltonian systems, Lagrangian systems, and the Legendre transformation},
\newblock Symposia Math.  {\bf 14} (1974),  247--258.


\bibitem{Tu6}
\newblock W. M. Tulczyjew,
\newblock \emph{The Legendre Transformation},
\newblock Ann. Inst. H. Poincar\'e Sect. A (N.S.), {\bf 27} (1977), 101--114.




\bibitem{Tu5}
\newblock W. M. Tulczyjew,
\newblock \emph{A symplectic framework for linear field theories},
\newblock Ann. Mat. Pura Appl., {\bf 130}
 (1982), 177--195.

\bibitem{Tu3}
\newblock W. M. Tulczyjew,
\newblock \emph{Geometric Formulation of Physical Theories},
\newblock Bibliopolis, Naples, (1989).

\bibitem{TU}
\newblock W. M. Tulczyjew, P. Urba\'nski,
\newblock \emph{A slow and careful Legendre transformation for singular Lagrangians},
\newblock The Infeld Centennial Meeting (Warsaw,1998), {\em Acta Phys. Polon. B},  {\bf 30} (1999), 2909--2978. 


\bibitem{V}
\newblock L. Vitagliano,
\newblock \emph{Partial Differential Hamiltonian Systems},
\newblock Cand. J. Math., {\bf 65} (2013), 1164--1200.

\bibitem{Ur}
\newblock P. Urba\'nski,
\newblock \emph{Double vector bundles in classical mechanics},
\newblock Rend. Sem. Matem. Torino, {\bf 54} (1996),  405--421.

\bibitem{Vor}
\newblock T.~T. Voronov,
\newblock \emph{Graded manifolds and Drienfeld doubles for Lie
bialgebroids},
\newblock in {"Quantization, Poisson brackets and beyond" (Manchester, 2001)}, {Contemp. Math.} {\bf 315}, Amer. Math. Soc., Providence, RI, 2002, pp. 131--168.

\bibitem{X}
\newblock Y.~Xin,
\newblock \emph{Minimal submanifolds and related topics},
\newblock Nankai Tracts in Mathematics {\bf 8}, World Scientific Publishing Co., Inc., River Edge, NJ, 2003.

\end{thebibliography}
\end{document}